\documentclass{article}

\usepackage{lineno,hyperref}
\modulolinenumbers[5]

\usepackage{graphicx}

\usepackage{amssymb}
\usepackage{amsthm}
\usepackage{amsmath}

\usepackage{caption}

\newtheorem{Thm}{Theorem}[section]
\makeatletter
\renewenvironment{proof}[1][\proofname]{\par
  \normalfont
  \topsep6\p@\@plus6\p@ \trivlist
  \item[\hskip\labelsep{\bfseries #1}\@addpunct{\bfseries.}]\ignorespaces
}{%
  \endtrivlist
}
\renewcommand{\proofname}{Proof}
\makeatother

\usepackage{latexsym}
\def\qed{\hfill $\Box$}

\usepackage{caption}

\newcommand{\argmin}{\mathop{\rm arg~min}\limits}

\def\vector#1{\mbox{\boldmath $#1$}}

\usepackage[top=1in, bottom=1in, left=1in, right=1in]{geometry}

\begin{document}

\title{GMRES using pseudoinverse for range symmetric singular systems}

\author{Kota Sugihara\footnote{kouta.sugihara@gmail.com}, \; Ken~Hayami\footnote{Professor Emeritus, National Institute of Informatics, and The Graduate University for Advanced Studies (SOKENDAI),  Email: hayami@nii.ac.jp}, \;  and \; Liao Zeyu\footnote{Department of Informatics, The Graduate University for Advanced Studies (SOKENDAI), 2-1-2, Hitotsubashi, Chiyoda-ku, Tokyo, 101-8430, Japan} }

\date{}

\maketitle




\begin{abstract}
Consider solving large sparse range symmetric 
singular linear systems $A\vector{x} = \vector{b}$ which arise, for instance, in the discretization 
of convection diffusion equations with periodic boundary conditions, and partial differential equations for electromagnetic fields using the edge-based finite element method.

In theory, the Generalized Minimal Residual (GMRES) method converges to the least squares solution for
inconsistent systems if the coefficient matrix $A$ is range symmetric, 
i.e. $ {\rm R}(A)= {\rm R}(A^{ \rm T } )$, where 
$ {\rm R}(A)$ is the range space of $A$.

We derived the necessary and sufficient conditions for GMRES to determine a least squares solution of inconsistent
and consistent range symmetric systems assuming exact arithmetic except for the computation of the elements
of the Hessenberg matrix.

In practice, GMRES may not converge due to numerical instability. In order to improve the convergence, 
we propose using the pseudoinverse
for the solution of the severely ill-conditioned Hessenberg systems in GMRES.
Numerical experiments on 
inconsistent systems indicate that the method is effective and robust.
Finally, we further improve the convergence of the method by reorthogonalizing the Modified Gram-Schmidt procedure. 
\end{abstract}

\vspace{5mm}

\section{INTRODUCTION}\label{sec:introsec}
Consider the system of linear equations
\begin{equation}\label{eqn:probset}
 A\vector{x} = \vector{b}
\end{equation}
or the linear least squares problem
\begin{equation}\label{eqn:probLS}
  \min_{\vector{x}\in \mathbb{R}^{n}}\|\vector{b} - A\vector{x}\|_{2}
\end{equation}
where $A\in \mathbb{R}^{n\times n}$ is range symmetric i.e. ${\rm R}(A)={\rm R}(A^{\rm T})$ and singular, \\
$\vector{x}, ~\vector{b}\in \mathbb{R}^{n}$,
which arise, for instance, in the discretization of convection diffusion equations with periodic boundary conditions \cite{Brown}, 
and partial differential equations of
electromagnetic fields using the edge-based finite element method \cite{edgMag,StaMag}.
(\ref{eqn:probset}) is called consistent 
when $\vector{b}\in R(A)$, and inconsistent otherwise.

The obvious Krylov subspace methods
for solving (\ref{eqn:probset}) would be
the Generalized Minimal Residual (GMRES) method \cite{Saad2nd,Saad} considering range symmetry 
${\rm R}(A)={\rm R}(A^{\rm T})$ of the
coefficient matrix $A$, which guarantees the convergence of
GMRES to 
a least squares solution of (\ref{eqn:probLS}) when $\vector{b} \notin {\rm R}(A)$
without breakdown \cite{Brown,hayamiM}.
However, for inconsistent systems, GMRES sometimes does not converge well numerically even if ${\rm R}(A)={\rm R}(A^{\rm T})$
since the condition number of the Hessenberg matrix becomes extremely large \cite{Brown,MoriEP}.

Assume that $\vector{b}$ is exact. That is, there are no discretization or measurement errors in $\vector{b}$.
In this paper, we prove that GMRES determines a least squares solution assuming exact arithmetic except
for the computation of the elements of the Hessenberg matrix, under certain conditions.
We also propose using pseudoinverse to solve the Hessenberg systems in GMRES in order to improve the numerical convergence  
for inconsistent systems.
Some numerical experiments on symmetric semidefinite inconsistent systems 
and nonsymmetric but range symmetirc singular systems
indicate that the method is effective and robust.

For some ill-conditioned and inconsistent systems, the convergence of GMRES using pseudoinverse is not enough.
For such cases, we show that the convergence may be improved by reorthogonalizing the modified Gram-Schmidt procedure.

We note that when $\vector{b}$ is contaminated by discretization error or measurement error, we may for instance use the
discrepancy principle and terminate the GMRES iterations when the residual is compatible to the error in the right-hand side.
In such a case, the Hessenberg matrix is not yet so ill-conditioned, so it is not necessary to use pseudoinverse to solve
the Hessenberg system.

\section{Motivation of this research}\label{sec:Motivation}
In this paper, we are addressing the problem of making GMRES converge for severely ill-conditioned 
or singular inconsistent systems. There are problems which are inconsistent even if we assume that 
there are no discretization errors or measerument errors. For example, 
in the 
partial differential equation \\
$\displaystyle {\rm curl}~\nu ({\rm curl} \vector{A}) = \vector{J}_{0}$
 for static magnatic fields \cite{StaMag},
the right hand side $\vector{J}_{0}$  may not satisfy 
$ \nabla \cdot \vector{J}_{0} = 0 $.
Here, ${\rm curl} \vector{A}$ is defined as $\nabla \times \vector{A}$, $\vector{A}$ is the vector potential,
$\nu$ is the magnetic reluctivity, and $\vector{J}_{0}$ is the external current density.
Then, if 
$\vector{J}_{0}$  does not satisfy $\displaystyle \nabla \cdot \vector{J}_{0} = 0$,
the linear system which arises by discretizing this partial equation becomes inconsistent even if 
there are no discretization errors.
After discretization, one could make the system (\ref{eqn:probset}) consistent by projecting $\vector{b}$
to ${\rm R}(A)$ in order that the (preconditioned) conjugate gradient (CG) converges to a solution.
However, in general, this may be infeasible if ${\rm R}(A)$ is not given explicitly.
Therefore, we consider solving the inconsistent system directly without transforming this system into a consistent system.
Thus, we use GMRES which is guaranteed to converge for inconsistent systems if ${\rm R}(A)={\rm R}(A^{\rm T})$.
Furthermore, we
propose using pseudoinverse to solve the Hessenberg systems 
and reorthogonalization of the Arnoldi process
in GMRES in order to 
improve the numerical convergence for inconsistent systems.

\section{GMRES}\label{sec:GMRESsec}
Let $\vector{x}_{0}$ be the initial approximate solution and 
$\vector{r}_{0} = \vector{b} - A\vector{x}_{0}$
be the initial residual vector. Denote 
the Krylov subspace by \\
$K_{k}(A, \vector{r}_{0}) = {\rm span}(\vector{r}_{0}, A\vector{r}_{0}, ...., A^{k-1} \vector{r}_{0})$. 
GMRES is an iterative method which finds an approximate solution $\vector{x}_{k}$ which 
satisfies
\begin{equation}\label{eq:algmres}
\vector{x}_{k} = \argmin_{\vector{x} \in \vector{x}_{0} + K_{k}(A, \vector{r}_{0})}\|\vector{b} - A\vector{x}\|_{2}
\end{equation}

Denote by $V_{k}$, the $n\times k$ matrix with column vectors $\vector{v}_{1},...,\vector{v}_{k}$ which
forms an orthonormal basis of 
$K_{k}(A, \vector{r}_{0})$.
An approximate solution $\vector{x}_{k} \in \vector{x}_{0} + K_{k}(A, \vector{r}_{0})$ can be obtained 
as 
$\vector{x}_{k}=\vector{x}_{0}+V_{k}\vector{y}_{k}$
where
\begin{equation}\label{eq:algmresH}
\vector{y}_{k} = \argmin_{\vector{y} \in  \mathbb{R}^{k}}\|\beta\vector{e}_{1} - H_{k+1,k}\vector{y}\|_{2}.
\end{equation}
Here, $H_{k+1,k} = [h_{i,j}] \in \mathbb{R}^{(k+1)\times k}$, where $AV_{k} = V_{k+1}H_{k+1,k}$ holds,\\ 
$\beta = \|\vector{r}_{0}\|_{2} = \|\vector{b}-A\vector{x}_{0}\|_{2}$ and 
$\vector{e}_{1}=[1,0,...,0]^{\rm{T}}$.

\section{Convergence analysis of GMRES considering rounding error for computing $h_{i,j}$}\label{sec:convStab}
Let $h_{i,j}$ be the $(i,j)$ element of $H_{k+1,k}$.

\begin{Thm}\label{Thm:StabGmres}
Let $u$ be the unit roundoff.
Let $\|H_{k+1,k}\|_{F}$ denote the Frobenius norm of $H_{k+1,k}$.
Assume exact arithmetic except for the computation of $h_{i,j}$.
Let ${\rm R}(A) = {\rm R}(A^{{\rm T}})$. 

Then, the following hold.
\begin{enumerate}
\item If (\ref{eqn:probset}) is inconsistent,
GMRES determines a solution of $\min_{\vector{x}\in \mathbb{R}^{n}}\|\vector{b} - A\vector{x}\|_{2}$ at the k-th step if and only if 
$h_{k+1,k}/\|H_{k,k}\|_{F} = O(\sqrt{u})$.
\item  If (\ref{eqn:probset}) is consistent,
GMRES determines a solution of $\min_{\vector{x}\in \mathbb{R}^{n}}\|\vector{b} - A\vector{x}\|_{2}$ at the k-th step if and only if 
$h_{k+1,k}/{\min_{1\leq i \leq k}{|h_{i,k}|}} = O(\sqrt{u})$,
where $\min_{1\leq i \leq k}{|h_{i,k}|}$ minimizes $|h_{i,k}|$ for $i$ such that $h_{i,k} \neq 0$. 
\end{enumerate}
\end{Thm} 

\begin{proof}

First, consider the case when (\ref{eqn:probset}) is inconsistent.
In the Arnoldi process,
\begin{eqnarray}\label{eq:eqArno}
AV_{k} =   V_{k+1}H_{k+1,k} = V_{k}H_{k,k}+h_{k+1,k}[0,..,0,\vector{v}_{k+1}]
\end{eqnarray}
holds.
Here, $\vector{v}_{k+1}$ is the ($k$+1)th column vector of $V_{k+1}$
and $\vector{v}_{1},....,\vector{v}_{k+1}$ are orthonormal.

From (\ref{eq:eqArno}), 
${\|AV_{k}\|_{F}}^{2}= {\|V_{k}H_{k,k}\|_{F}}^{2} + 2(\sum_{i=1}^{k}h_{i,k}\vector{v}_{k}, \vector{v}_{k+1})
+ {h_{k+1,k}}^{2}{\|\vector{v}_{k+1}\|_{2}}^{2} ={\|V_{k}H_{k,k}\|_{F}}^{2} + {h_{k+1,k}}^{2}$ holds
since $\vector{v}_{k+1}$ is orthogonal to
all columns of $V_{k}$ and ${\|\vector{v}_{k+1}\|_{2}}^{2}=1$.

If $\frac{{h_{k+1,k}}^{2}}{{\|H_{k,k}\|_{F}}^{2}} = O(u)$ holds,
then we may regard $\frac{{h_{k+1,k}}^{2}}{{\|H_{k,k}\|_{F}}^{2}} \approx 0 $ in finite precision arithmetic.
Then, ${\|AV_{k}\|_{F}}^{2} \approx {\|V_{k}H_{k,k}\|_{F}}^{2}$ holds, 
since ${\|V_{k}H_{k,k}\|_{F}}^{2} = {\|H_{k,k}\|_{F}}^{2}$.
Hence, $AV_{k} \approx V_{k}H_{k,k}$ holds in finite precision arithmetic.

Refer to the proof of Theorem 1 in \cite{RMINRES}.
In the present proof, the preconditioner $M$ is an identity matrix and MINRES is replaced by GMRES.
The upper triangular matrix $T_{j}$ is replaced by the Hessenberg matrix $H_{j,j}$.
In order to prove the theorem, we will analyse GMRES
by decomposing it into the ${\rm R}(A)$ component and the ${{\rm R}(A)}^{\perp}$ component.
Using the approach in the proof of Theorem 1 in \cite{RMINRES}, 
we can prove that the ${\rm R}(A)$ component $\vector{x}_{1}^{k}$ of the $k$th iterate 
$ { \vector{x} }^{k} $ of GMRES
minimizes the $ { \rm R }(A) $ component of $ \| \vector{b} - A \vector{x} \|_{2} $ 
when $\frac{h_{k+1,k}}{\|H_{k+1,k}\|_{F}} = O(\sqrt{u})$ holds.
Hence, we can prove that the $k$th iterate ${\vector{x}}^{k}$ of GMRES
minimizes $ \| \vector{b} - A\vector{x} \|_{2} $ when $\frac{h_{k+1,k}}{\|H_{k+1,k}\|_{F}} = O(\sqrt{u})$ holds.

Now assume that GMRES determines a solution of 
$\min_{\vector{x}\in \mathbb{R}^{n}}\|\vector{b} - A\vector{x}\|_{2}$ at the $k$th step.
From Theorem 2.4 in \cite{Brown}, ${\rm rank}AV_{k} < {\rm rank}V_{k} = k$ and ${\rm rank}AV_{k-1} = k-1$ if 
(\ref{eqn:probset}) is inconsistent.
From (\ref{eq:eqArno}), $V_{k+1}H_{k+1,k}$ is rank-deficient.
Therefore, there exists $\vector{x} \neq \vector{0}$ such that $V_{k+1}H_{k+1,k}\vector{x} = \vector{0}$.

Let $x_{k}$ be the $k$th element of $\vector{x}$.
From (\ref{eq:eqArno}),\\
 $V_{k+1}H_{k+1,k}\vector{x} = V_{k}H_{k,k}\vector{x} + h_{k+1,k}x_{k}\vector{v}_{k+1} = \vector{0}$.
Hence,\\
 ${\vector{x}}^{\rm T}{H_{k,k}}^{\rm T}H_{k,k}\vector{x} + {h_{k+1,k}}^{2}{x_{k}}^{2} = 0$.
If $x_{k} = 0$,
 $V_{k}H_{k,k}\vector{x} = \vector{0}$. Since ${\rm rank}V_{k} = k$, $H_{k,k}\vector{x} = \vector{0}$.
Since $x_{k} = 0$, the first $k-1$ column vectors of $H_{k,k}$ are linearly dependent.
However, all columns of $H_{k,k-1}$ are linearly independent since ${\rm rank}AV_{k-1} = k-1$ and $AV_{k-1} = V_{k}H_{k,k-1}$.
This is a contradiction. Thus, $x_{k} \neq 0$.
\begin{eqnarray*}\label{eq:proofeq}
{\vector{x}}^{\rm T}{H_{k,k}}^{\rm T}H_{k,k}\vector{x} + {h_{k+1,k}}^{2}{x_{k}}^{2} 
 & = &  {x_{k}}^{2}(\frac{{\|H_{k,k}\vector{x}\|_{2}}^{2}}{{x_{k}}^{2}} + {h_{k+1,k}}^{2}) \\
 & \geq & {x_{k}}^{2}(\frac{{\|H_{k,k}\vector{x}\|_{2}}^{2}}{{\|\vector{x}\|_{2}}^{2}} + {h_{k+1,k}}^{2})
\end{eqnarray*}
Assume $\frac{h_{k+1,k}}{\|H_{k,k}\|_{F}} > O(\sqrt{u})$. Then, ${h_{k+1,k}}^{2} > 0$.
If $\|H_{k,k}\vector{x}\|_{2} = 0$, then \\
${x_{k}}^{2}(\frac{{\|H_{k,k}\vector{x}\|_{2}}^{2}}{{\|\vector{x}\|_{2}}^{2}} + {h_{k+1,k}}^{2}) > 0$ since
${h_{k+1,k}}^{2} > 0$.
If $\|H_{k,k}\vector{x}\|_{2} > 0$, then ${x_{k}}^{2}(\frac{{\|H_{k,k}\vector{x}\|_{2}}^{2}}{{\|\vector{x}\|_{2}}^{2}} + {h_{k+1,k}}^{2}) > 0$.
Then, ${\vector{x}}^{\rm T}{H_{k,k}}^{\rm T}H_{k,k}\vector{x} + {h_{k+1,k}}^{2}{x_{k}}^{2} > 0$.
This is a contradiction.
Thus, 
 $\frac{h_{k+1,k}}{\|H_{k,k}\|_{F}} = O(\sqrt{u})$.

Next, we will prove the theorem for the singular consistent system.
In order to prove the theorem for the consistent system, we will analyze
GMRES by decomposing it into the ${\rm R}(A)$ component and the ${\rm R}(A)^{\perp}$ component.
Using the same approach as \cite{hayamiM,RMINRES}, the ${\rm R}(A)$ component of the decomposed 
GMRES for the consistent system is equivalent to GMRES applied to a nonsingular system.
Here, we let the nonsingular system be $A_{11}\tilde{\vector{x}} = \vector{b}_{1}$.
Furthermore,
the ${\rm R}(A)^{\perp}$ components of GMRES are $\vector{0}$ when the initial vector $\vector{x}_{0} = \vector{0}$.
We will refer to the proof of Proposition 6.10 in \cite{Saad}
which proves
that $h_{j+1,j} = 0$ if and only if the approximate solution $\vector{x}_{j}$ of GMRES is exact for nonsingular systems.
Assume that $\frac{h_{k+1,k}}{\min_{1\leq i \leq k}|h_{i,k}|} = O(\sqrt{u})$ holds.
As in \cite{Saad}, the scalars $c_{i}$ 
and $s_{i}$ of the $i$th Givens rotation $\Omega_{i}$ 
are defined as $s_{i} = \frac{h_{i+1,i}}{\sqrt{(h_{i,i}^{(i-1)})^{2} + {h_{i+1,i}}^{2}}}$, \\
$c_{i} = \frac{h_{i,i}^{(i-1)}}{\sqrt{(h_{i,i}^{(i-1)})^{2} + {h_{i+1,i}}^{2}}}$
where  ${h_{i,i}}^{(i-1)}$ is a linear combination of $h_{1,i},h_{2,i},...,h_{i,i}$
by $\Omega_{i-1}\Omega_{i-2}...\Omega_{1}$. 
Since the ${\rm R}(A)$ component of the decomposed 
GMRES is equivalent to GMRES applied to a nonsingular system, then $r_{k,k} = h_{k,k}^{(k-1)}$
is nonzero by the first part of Proposition 6.9 in \cite{Saad}.
Then, 
there exists a positive scalar $\alpha_{k}$ which satisfies the following inequality.
Here, $\alpha_{k}$ is independent of $h_{k+1,k}$.
\begin{eqnarray*}
s_{k} & = & \frac{h_{k+1,k}}{\sqrt{(h_{k,k}^{(k-1)})^{2} + {h_{k+1,k}}^{2}}} \\
       & \leq & \alpha_{k} \times \frac{h_{k+1,k}}{\min_{1 \leq i \leq k}|h_{i,k}|} \\
       & = & O(\sqrt{u})
\end{eqnarray*}
Since $s_{k} = O(\sqrt{u})$, we may regard ${s_{k}}^{2} \approx 0$ in finite precision arithmetic. \\
Then, the relation ${\|\vector{b}_{1} - A_{11}\tilde{\vector{x}}_{k}\|_{2}}^{2} 
= {s_{k}}^{2}{\|\vector{b}_{1} - A_{11}\tilde{\vector{x}}_{k-1}\|_{2}}^{2}$
implies that ${\|\vector{b}_{1} - A_{11}\tilde{\vector{x}}_{k}\|_{2}}^{2} = 0$.\\
Therefore,  ${\|\vector{b} - A\vector{x}_{k}\|_{2}}^{2} = {\|\vector{b}_{1} - A_{11}\tilde{\vector{x}}_{k}\|_{2}}^{2} = 0$.

Now assume that GMRES determines a solution of $\min_{\vector{x}\in \mathbb{R}^{n}}\|\vector{b} - A\vector{x}\|_{2}$ at the $k$th step.
We will prove by contradiction that $\frac{h_{k+1,k}}{\min_{1\leq i \leq k}{|h_{i,k}|}} = O(\sqrt{u})$ holds.
Assume $\frac{h_{k+1,k}}{\min_{1\leq i \leq k}{|h_{i,k}|}} > O(\sqrt{u})$.
Then, $\vector{v}_{k+1}$ of the orthonormal basis exists.
Since (\ref{eqn:probset}) is consistent, there exists a nonzero vector $\vector{y}_{k}$ and the $k$th element
$y_{k}^{(k)} \neq 0$ which satisfies
$\|\beta\vector{v}_{1} - A[\vector{v}_{1},...,\vector{v}_{k}]\vector{y}_{k}\|_{2} = 0$. Here, 
$\beta = \|\vector{b} - A\vector{x}_{0}\|_{2}$ where $\vector{x}_{0}$ is an initial solution vector.
\begin{eqnarray*}
|(\beta\vector{v}_{1} - A[\vector{v}_{1},...,\vector{v}_{k}]\vector{y}_{k}, \vector{v}_{k+1})| & = & 
h_{k+1,k}|y_{k}^{(k)}|{\|\vector{v}_{k+1}\|_{2}}^{2} \\
& = & h_{k+1,k}|y_{k}^{(k)}| \\
& > & |y_{k}^{(k)}|\min_{1 \leq i \leq k}|h_{i,k}|O(\sqrt{u})
\end{eqnarray*}
Since $y_{k}^{(k)} \neq 0$ and $\min_{1\leq i \leq k}|h_{i,k}| > 0$, then 
$|y_{k}^{(k)}|\min_{1 \leq i \leq k}|h_{i,k}|O(\sqrt{u}) > 0$.\\
However, $\|\beta\vector{v}_{1} - A[\vector{v}_{1},...,\vector{v}_{k}]\vector{y}_{k}\|_{2} = 0$.
This is a contradiction.
Thus, $\frac{h_{k+1,k}}{\min_{1\leq i\leq k}|h_{i,k}|} = O(\sqrt{u})$.

\qed\end{proof}

If all computations are done in exact arithmetic, GMRES determines a solution 
of $\min_{\vector{x}\in \mathbb{R}^{n}}\|\vector{b} - A\vector{x}\|_{2}$ when $h_{k+1,k}=0$.
When $h_{k+1,k}=0$ holds, $H_{k,k}$ is singular (See \cite{RMINRES}, Theorem 1, point a, b; \cite{HS}, Theorem 4).

On the other hand, in Theorem \ref{Thm:StabGmres}, 
when $h_{k+1,k}/\|H_{k,k}\|_{F} > O(\sqrt{u})$ for the incosistent systems or 
$h_{k+1,k}/{\min_{1\leq i \leq k}{|h_{i,k}|}} > O(\sqrt{u})$ for the consistent systems,
GMRES does not converge, whereas when $h_{k+1,k}/\|H_{k,k}\|_{F} = O(\sqrt{u})$ for the incosistent systems or $h_{k+1,k}/{\min_{1\leq i \leq k}{|h_{i,k}|}} = O(\sqrt{u})$ for the consistent systems, GMRES converges to a least squares solution of (\ref{eqn:probLS}).

However, numerical experiments in Sections \ref{sec:NumerStab} and \ref{sec:StabReorth} for inconsistent systems indicate that
$\displaystyle \frac{\|A^{\rm T}\vector{r}\|_{2}}{\|A^{\rm T}\vector{b}\|_{2}}$ becomes very small when the smallest singular value of 
$H_{k+1,k}$ is very small, but $h_{k+1,k}$ is not small unlike in Theorem \ref{Thm:StabGmres}.
We think the numerical result concerning $h_{k+1,k}$ is different from Theorem \ref{Thm:StabGmres} due to rounding errors when
$\displaystyle \frac{\|A^{\rm T}\vector{r}\|_{2}}{\|A^{\rm T}\vector{b}\|_{2}}$ becomes very small. 
This is because Theorem \ref{Thm:StabGmres} takes rounding errors in to account only for the computation of $h_{i,j}$ and explains the relation between 
$\displaystyle \frac{\|A^{\rm T}\vector{r}\|_{2}}{\|A^{\rm T}\vector{b}\|_{2}}$ 
and $h_{k+1,k}$. That is, corresponding to the convergence theory of GMRES in \cite{Brown}, 
Theorem \ref{Thm:StabGmres} is the convergence theory considering rounding errors 
only for the computation of $h_{i,j}$.

\section{GMRES USING PSEUDOINVERSE}\label{sec:Stabsec}
In finite precision arithmetic, the backward substitution of GMRES does not work well 
when $H_{k+1,k}$ becomes severely ill-conditioned.
Therefore, the GMRES solution is inaccurate when $H_{k+1,k}$ becomes severely ill-conditioned.
Thus, we will propose GMRES using pseudoinverse in order to improve 
the accuracy of the GMRES solution.

Assume that $R(A) = R(A^{{\rm T}})$ holds. Consider inconsistent systems 
where \\ 
$\vector{b} \notin R(A)$ in (\ref{eqn:probset}).
GMRES converges to 
a least squares solution without breakdown at some step, then GMRES breaks down at the 
next step, with breakdown through rank deficiency of the least squares problems \cite{Brown,hayamiM}. 
Rank deficiency of the least squares problems means the Hessenberg matrix is rank deficient\cite{Brown}.
Rank deficiency of the Hessenberg matrix means that the smallest singular value $\sigma_{k}(H_{k+1,k})$ is 0.
Hence, numerically, the condition number of the Hessenberg matrix $H_{k+1,k} \in \mathbb{R}^{(k+1)\times k}$ in (\ref{eq:algmresH})
becomes extremely large (~$O(\frac{1}{u})$ where $u$ is the unit roundoff \cite{Hig}) for inconsistent systems 
when GMRES converges to a least squares solution. (See Fig. \ref{fig:Ar-sigma-h-pinv-ms},
\ref{fig:Ar-sigma-h-pinv-32}, \ref{fig:Ar-sigma-h-pinv-re2-ms}, 
\ref{fig:Ar-sigma-h-pinv-re2-32}.)
We apply Givens rotation to $\min_{\vector{y} \in  \mathbb{R}^{k}}\|\beta\vector{e}_{1} - H_{k+1,k}\vector{y}\|_{2}$.
Then the upper triangular system $R_{k}\vector{y}=\vector{g}_{k}$
is generated. Since the condition number
$\kappa(H_{k+1,k})=\kappa(R)$, if the condition number of $H_{k+1,k}$ is too large, then
the backward substitution for $R_{k}\vector{y}=\vector{g}_{k}$ does not work well due to rounding errors.
Hence, GMRES does not converge well.

In order to solve this difficulty,
we propose using pseudoinverse for solving (\ref{eq:algmresH}) as follows.

\vspace{6pt}
\hspace{-13pt}{\bf Algorithm 1 : GMRES using pseudoinverse (essence) }
\vspace{4pt}
\\1: Compute $y = {H_{k+1,k}}^{\dag}\beta\vector{e}_{1}$ where
${H_{k+1,k}}^{\dag}$ is the pseudoinverse of $H_{k+1,k}$.
\\2: Compute the solution $\vector{x}_{k}=\vector{x}_{0}+V_{k}y$.

Here,  $y = {H_{k+1,k}}^{\dag}\beta\vector{e}_{1}$ is the minimum-norm solution 
of
$\min_{\vector{y}_{k} \in  \mathbb{R}^{k}}\|\beta\vector{e}_{1} - H_{k+1,k}\vector{y}_{k}\|_{2}$ \cite{Bjorck}.  

${H_{k+1,k}}^{\dag}$ is defined as follows.

\vspace{6pt}
\hspace{-13pt}{\bf Definition 2 : Pseudoinverse of $B$}
\vspace{4pt}
\\1: Let the singular value decomposition of $B$ be $B = U\Sigma V^{\rm T}$ where $U\in {\rm \mathbb{R}}^{m\times m}$ and 
$V \in {\rm \mathbb{R}}^{n\times n}$ are orthogonal matrices, $\Sigma \in {\rm \mathbb{R}}^{m\times n}$ is
the diagonal matrix whose diagonal elements are the singular values
$\sigma_{1} \geq ...\geq \sigma_{r} > 0$, $r = {\rm rank}B$, $\sigma_{i} = 0$, $i = r + 1, ..., \min\{m,n\}$. 
\\2: Then, $B^{\dag} = V{{\Sigma}^{\dag}}U^{\rm T}$.
Here, 
$\Sigma^{\dag} \in {\rm \mathbb{R}}^{n\times m}$ is the diagonal matrix whose diagonal elements are\\
${\sigma_1}^{-1} \leq ... \leq {\sigma_r}^{-1}$, ${\sigma_{i}}^{\dag} = 0$, $i = r + 1, ..., \min\{m,n\}$.

We use pinv in MATLAB for computing the pseudoinverse.
pinv for the matrix $B \in {\rm \mathbb{R}}^{m\times n}$ is defined as follows.

\vspace{6pt}
\hspace{-13pt}{\bf Algorithm 3 : pinv in MATLAB }
\vspace{4pt}
\\1:  Let the singular value decomposition of $B$ be $B = U{\Sigma}V^{\rm T}$ as above.
\\2: Set the tolerance value $tol$. The diagonal elements of $\Sigma$ which are smaller than $tol$ are replaced by zero
to give
\begin{eqnarray*}\label{eq:Asvd1}
\left[
\begin{array}{cc}
\Sigma_{1} & 0 \\
0                                 &  0                            
\end{array}
\right].
\end{eqnarray*}
Then, let
\begin{eqnarray}\label{eq:Asvd}
\tilde{B} := [U_{1},U_{2}]\left[
\begin{array}{cc}
\Sigma_{1} & 0 \\
0                                 &  0                            
\end{array}
\right][V_{1},V_{2}]^{\rm T}
=U_{1}{{\Sigma}_{1}}{V_{1}}^{\rm T}.
\end{eqnarray} 
where $ U=[U_{1},U_{2}], ~V=[V_{1},V_{2}] $.
\\3: $ \tilde{B}^{\dag} :=V_{1}{\Sigma_{1}}^{-1}{ U_{1} }^{\rm T}$.

In {\bf Algorithm 3}, the default value of the tolerance value $tol$ is \\
$\max\{m,n\}\times {\rm eps}(\|B\|_{2})$ for $B \in {\rm \mathbb{R}}^{m\times n}$.
Here,
\begin{itemize}
\item $d = {\rm eps}(x)$, where $x$ has data type single or double, returns the positive 
distance $d$ from $|x|$ to the next larger floating-point number of the same precision as $x$.   
\end{itemize}

$\max\{m,n\}\times {\rm eps}(\|B\|_{2})$ is called the numerical rank \cite{Bjorck}.

Here, let $ \sigma_{1} ( H_{ k+1,k } ) $ be the largest singular value of $ H_{k+1,k} $, and 
$ \sigma_{k} ( H_{k+1,k} ) $ be the smallest singular value of $H_{k+1,k}$.
Table~\ref{table:cond_of_H} indicates the condition number of $H_{k+1,k}$ and ${\tilde{H}_{k+1,k}}^{~~~~~~\dag}$.

\begin{table}[htbp]
\begin{center}
\caption{Condition number of $H_{k+1,k}$ and ${\tilde{H}_{k+1,k}}^{~~~~~~\dag}$}
\label{table:cond_of_H}
\begin{tabular}{|c|r|r|r|r|r|r|}
\hline
Matrix & Condition number \\
\hline
\hline
$H_{k+1,k}$ &  $\frac{\sigma_{1}(H_{k+1,k})}{\sigma_{k}(H_{k+1,k})}$  \\
\hline
${\tilde{H}_{k+1,k}}^{~~~~~~\dag}$ & $\frac{\sigma_{1}(H_{k+1,k})}{tol}$ \\
\hline
\end{tabular}
\end{center}
\end{table}

As $k$ increases, $\sigma_{k}(H_{k+1,k})$ decreases. Hence,
the condition number of $H_{k+1,k}$, i.e. $\frac{\sigma_{1}(H_{k+1,k})}{\sigma_{k}(H_{k+1,k})}$ may become too large.
Thus, the backward substitution for $R_{k}\vector{y}=\vector{g}_{1}$ may not work well since the condition number
$\frac{\sigma_{1}(H_{k+1,k})}{\sigma_{k}(H_{k+1,k})}$
is too large. On the other hand, 
if we truncate the singular values which are smaller than $tol$ using pinv of $H_{k+1,k}$, 
since $\sigma_{k}(H_{k+1,k})$ is smaller than $tol$, we truncate $\sigma_{k}(H_{k+1,k})$.
Then, $\frac{\sigma_{1}(H_{k+1,k})}{tol}$ is smaller than $\frac{\sigma_{1}(H_{k+1,k})}{\sigma_{k}(H_{k+1,k})}$. Hence, 
GMRES using pseudoinverse becomes more stable than GMRES.

\section{NUMERICAL EXPERIMENTS ON EVALUATION OF GMRES USING PSEUDOINVERSE}\label{sec:NumerStab}
In this section, we evaluate the effectiveness of GMRES using pseudoinverse for range symmetric 
singular systems.
To do so, we compare the performance and the convergence of GMRES using pseudoinverse,
GMRES and Range Restricted GMRES (RRGMRES)\cite{CLR1,Ry} (See also \cite{RRMINRES,ReiNeu}.) by numerical experiments.

We compare GMRES using pseudoinverse with RRGMRES since RRGMRES
works better than GMRES for inconsistent range symmetric systems.
The initial approximate vector is set to $\vector{0}$.
We evaluate the performance of each method by $\displaystyle \frac{\|A^{\rm T}\vector{r}\|_{2}}{\|A^{\rm T}\vector{b}\|_{2}}$
where $\vector{r}=\vector{b}-A\vector{x}_{k}$ and $\vector{x}_{k}$ is an approximate solution 
at the $k$th step.

Computation except for {\bf Algorithm 1} of GMRES using pseudoinverse
were done on a PC with Intel(R) Core(TM) i7-7500U 2.70 GHz CPU, Cent OS and double precision
floating arithmetic.
GMRES and RRGMRES were coded in Fortran 90 and compiled by Intel Fortran.
The method to code GMRES using pseudoinverse is as follows.
Here, $H_{i,j}$ is the Hessenberg matrix and all the column vectors of $V_{k}$ form an orthonormal basis
generated by the Arnoldi process.
\begin{enumerate}
\item $H_{i,j}$ and $V_{k}$ are computed by Fortran 90.
\item Write $H_{i,j}$ and $V_{k}$ into the ascii formatted files by Fortran 90.
\item Read the files of $H_{i,j}$ and $V_{k}$ in MATLAB.
\item The pseudoinverse 
${\tilde{H}_{i,j}}^{~~\dag}$
and the solution 
$ \vector{x}_{k} = \vector{x}_{0} + V_{k} { {\tilde{H} }_{k+1,k}}^{~~~~~~\dag}\beta\vector{e}_{1}$ 
are computed using pinv of MATLAB.
\end{enumerate} 
The version of MATLAB is R2018b.

\subsection{GMRES USING PSEUDOINVERSE FOR SYMMETRIC MATRICES}\label{sec:pinvsym}
We will first use symmetric numerical positive semidefinite matrices from \cite{Florida}.
The information on these matrices is described in Table~\ref{table:matinfo}.
Here, $n$ and $nnz$ are the dimension and the number of
nonzero elements of the matrices, respectively.
rank$A$, $\kappa(A)$ are 
the dimension of ${\rm R}(A)$
and the condition number (the ratio of the maximum
singular value divided by the minimum singular value of) $A$, respectively.
They were computed by the function {\bf rank}
and {\bf svd} of MATLAB, respectively.

\begin{table}[htbp]
\begin{center}
\caption{Characteristics of the coefficient matrices of the test problems}
\label{table:matinfo}
\begin{tabular}{|c|r|r|r|r|r|}
\hline
Matrix & n & nnz &rank $A$ & $\kappa(A)$ & Application area \\
\hline
\hline
msc01050 &  1,050 & 26,198 & 1,049 & $ 8.997 \times 10^{15} $ & structural problem \\
\hline
ex32 & 1,159 & 11,047  & 1,158   & $ 1.3546 \times 10^{18} $ & CFD \\
\hline
\end{tabular}
\end{center}
\end{table}

For the above two matrices,
the right hand side vectors $\vector{b}$ were set as follows, where $\vector{b}_{\rm{N}(A)}$ 
is a unit eigenvector corresponding to the smallest eigenvalue of $A$.
\begin{itemize}
\item $\vector{b} = \frac{A \times (1,1,.,1)^{\rm{T}}}{\|A \times (1,1,..,1)^{\rm{T}}\|_{2}} + \vector{b}_{\rm{N}(A)}\times 0.01$
\end{itemize}
Thus, the systems are inconsistent.

For symmetric singular systems, Minimal Residual (MINRES)\cite{Minres} 
and Range Restricted MINRES\\(RRMINRES)\cite{RRMINRES} (See also \cite{ReiNeu,DMR}.) methods  
should converge to a least squares solution in exact arithmetic. However, in finite precision arithmetic, 
they show ill-convergence for inconsistent systems as seen in 
Fig.  \ref{fig:Ar-mi-ms}, \ref{fig:Ar-rrmi-ms}, \ref{fig:Ar-mi-32} and \ref{fig:Ar-rrmi-32}.
This is because MINRES and RRMINRES use short-term recurrence, 
and are affected by rounding errors, especially for ill-conditioned inconsistent systems. 
GMRES and RRGMRES are more robust as seen in Fig. \ref{fig:Ar-pinv-gm-rr-ms}
and \ref{fig:Ar-pinv-gm-rr-32},
since they use full orthogonalization of the Arnoldi process. 
The contribution of the present paper is to make GMRES even more robust for ill-conditioned, inconsistent systems.
We will also report numerical results of MINRES and RRMINRES for the same symmetric singular systems.

Fig. \ref{fig:Ar-pinv-gm-rr-ms} for {\bf msc01050}
and 
Fig. \ref{fig:Ar-pinv-gm-rr-32} for {\bf ex32} 
show $\displaystyle \frac{\|A\vector{r}_{j}\|_{2}}{\|A\vector{b}\|_{2}}$ versus 
the iteration number for GMRES using pseudoinverse (blue), GMRES (red) and 
RRGMRES (green) for inconsistent problems. (Note $A^{\rm T} = A$ for these problems)

Fig. \ref{fig:Ar-sigma-h-pinv-ms} for {\bf msc01050}
and 
Fig. \ref{fig:Ar-sigma-h-pinv-32} for {\bf ex32} 
show $\displaystyle \frac{\|A\vector{r}_{j}\|_{2}}{\|A\vector{b}\|_{2}}$  (blue),
$\frac{\sigma_{k}(H_{k+1,k})}{\sigma_{1}(H_{k+1,k})}$ (red) and 
$\frac{h_{j+1,j}}{\|H_{j,j}\|_{F}}$ (green) versus 
the iteration number for GMRES using pseudoinverse for inconsistent problems.

Fig. \ref{fig:Ar-mi-ms} for {\bf msc01050}
and
Fig. \ref{fig:Ar-mi-32} for {\bf ex32} 
show $\displaystyle \frac{\|A\vector{r}_{j}\|_{2}}{\|A\vector{b}\|_{2}}$ versus 
the iteration number for MINRES for inconsistent problems.

Fig. \ref{fig:Ar-rrmi-ms} for {\bf msc01050}
and 
Fig. \ref{fig:Ar-rrmi-32} for {\bf ex32}
show $\displaystyle \frac{\|A\vector{r}_{j}\|_{2}}{\|A\vector{b}\|_{2}}$ versus 
the iteration number for RRMINRES for inconsistent problems.

\begin{figure}[htbp]
\begin{minipage}{0.5\hsize}
\begin{center}
\includegraphics[keepaspectratio,scale=0.4]{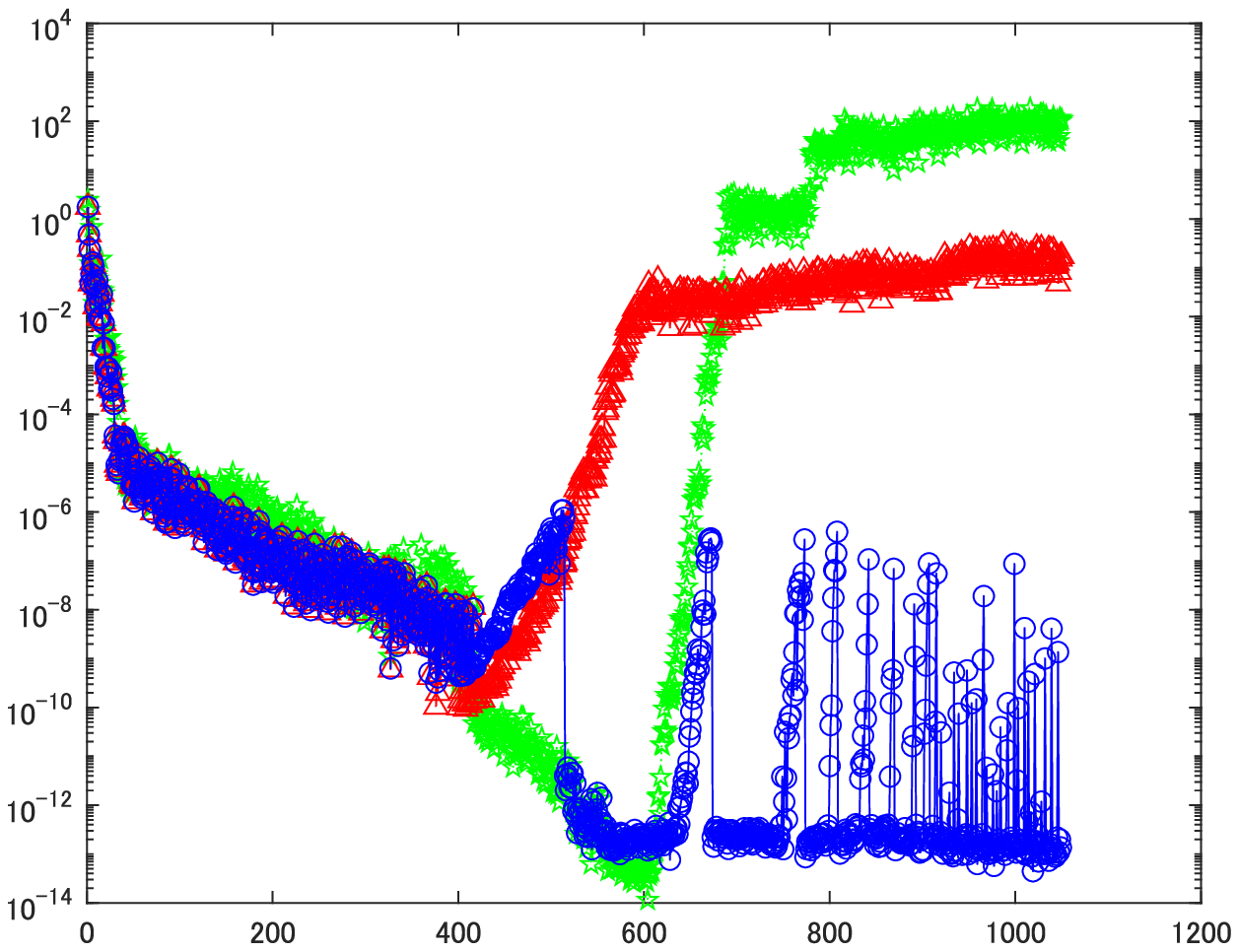}
\end{center}
\captionsetup{width=.95\linewidth}
\caption{$\displaystyle \frac{\|A\vector{r}_{j}\|_{2}}{\|A\vector{b}\|_{2}}$ vs. number of iterations for 
GMRES using pseudoinverse (blue), 
GMRES (red), 
and RRGMRES (green) for an inconsistent problem ({\bf msc01050})}
\label{fig:Ar-pinv-gm-rr-ms}
\end{minipage}
\begin{minipage}{0.5\hsize}
\begin{center}
\includegraphics[keepaspectratio,scale=0.4]{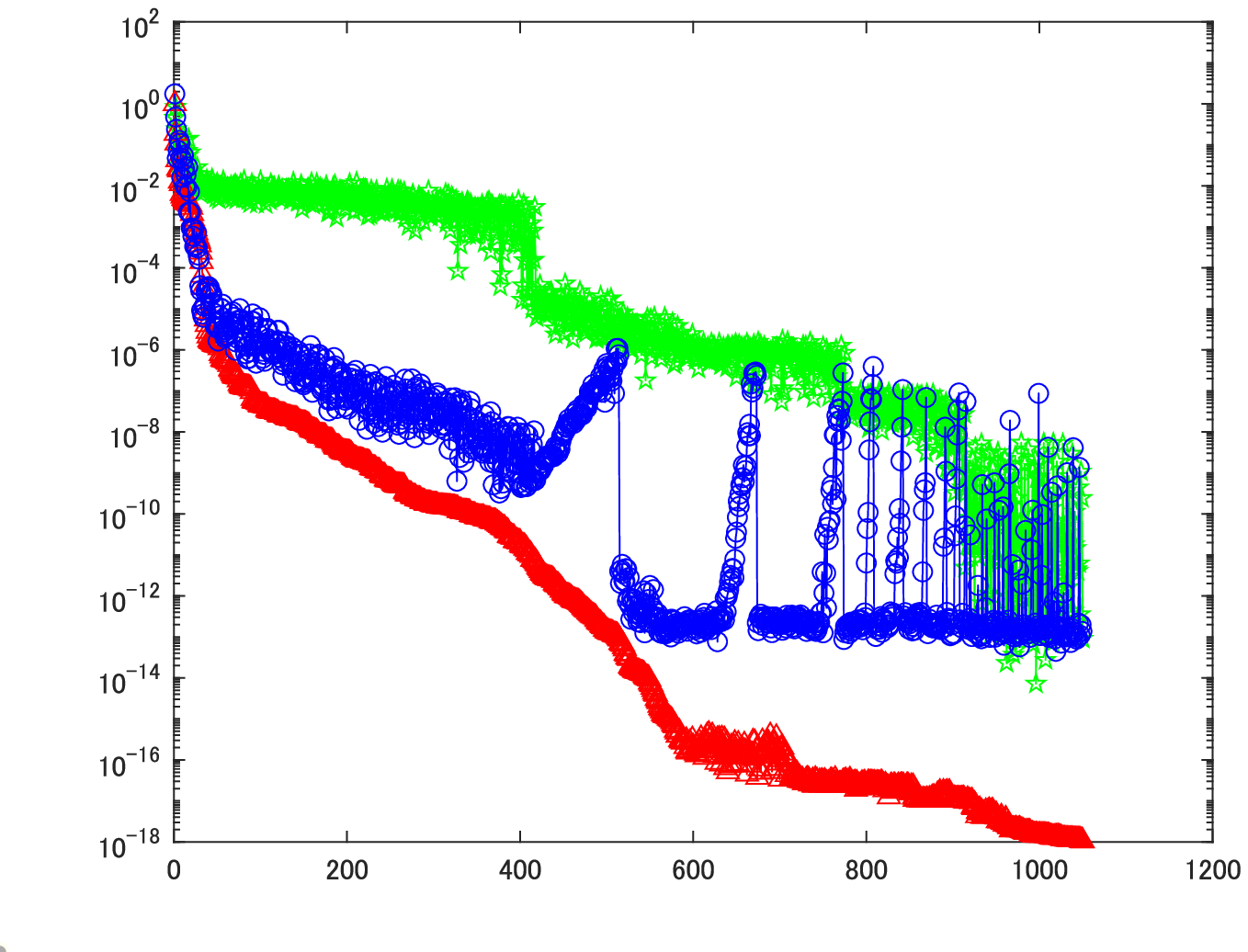}
\end{center}
\captionsetup{width=.95\linewidth}
\caption{$\displaystyle \frac{\|A\vector{r}_{j}\|_{2}}{\|A\vector{b}\|_{2}}$ (blue), 
$\frac{\sigma_{k}(H_{k+1,k})}{\sigma_{1}(H_{k+1,k})}$ (red) and 
$\frac{h_{j+1,j}}{\|H_{j,j}\|_{F}}$ (green) vs. number of iterations for 
GMRES using pseudoinverse for an inconsistent problem ({\bf msc01050})}
\label{fig:Ar-sigma-h-pinv-ms}
\end{minipage}
\end{figure}

\begin{figure}[htbp]
\begin{minipage}{0.5\hsize}
\begin{center}
\includegraphics[keepaspectratio,scale=0.4]{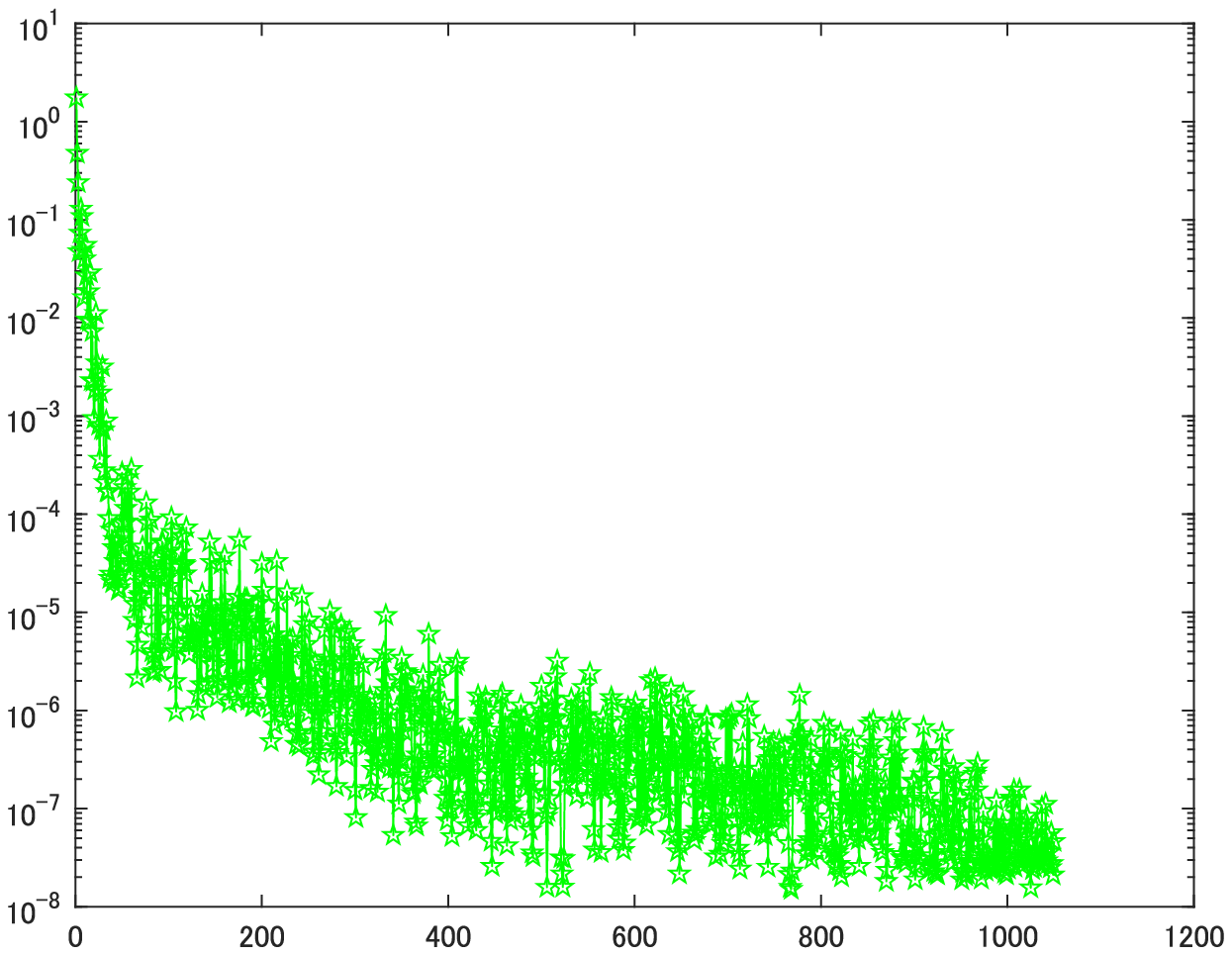}
\end{center}
\captionsetup{width=.95\linewidth}
\caption{$\displaystyle \frac{\|A\vector{r}_{j}\|_{2}}{\|A\vector{b}\|_{2}}$ vs. number of iterations for MINRES for an inconsistent problem ({\bf msc01050})}
\label{fig:Ar-mi-ms}
\end{minipage}
\begin{minipage}{0.5\hsize}
\begin{center}
\includegraphics[keepaspectratio,scale=0.4]{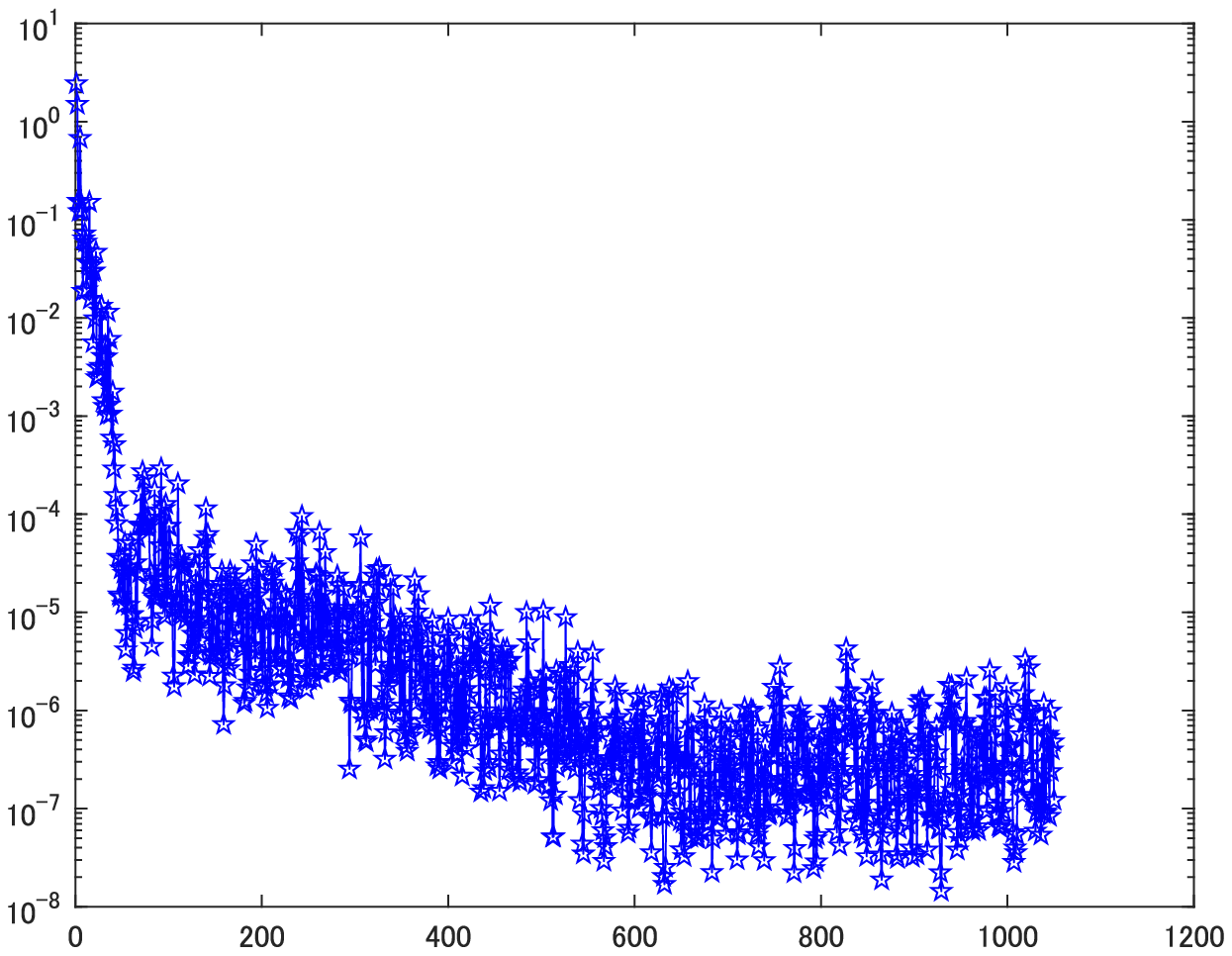}
\end{center}
\captionsetup{width=.95\linewidth}
\caption{$\displaystyle \frac{\|A\vector{r}_{j}\|_{2}}{\|A\vector{b}\|_{2}}$ vs. number of iterations for RRMINRES for an inconsistent problem ({\bf msc01050})}
\label{fig:Ar-rrmi-ms}
\end{minipage}
\end{figure}

\begin{figure}[htbp]
\begin{minipage}{0.5\hsize}
\begin{center}
\includegraphics[keepaspectratio,scale=0.4]{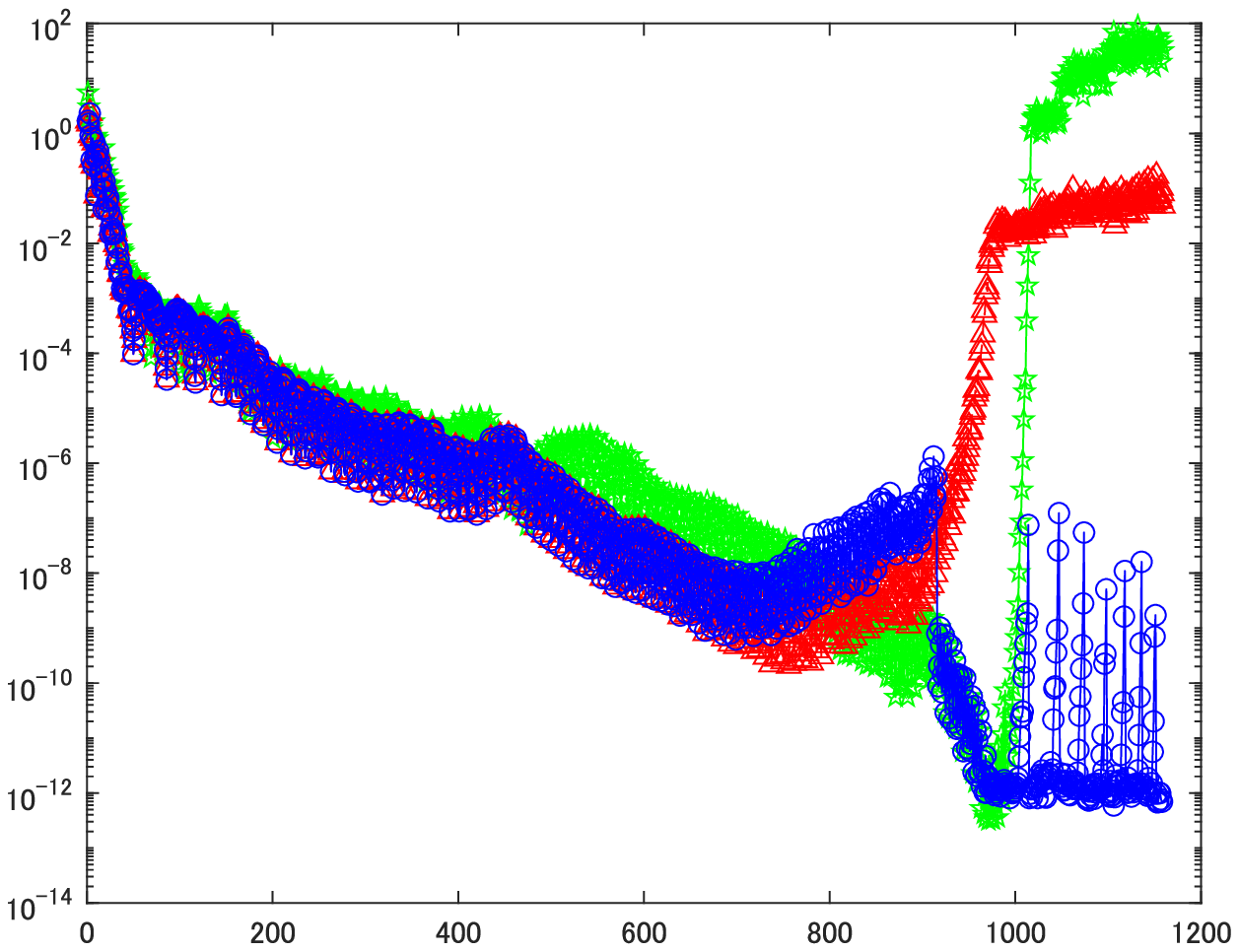}
\end{center}
\captionsetup{width=.95\linewidth}
\caption{$\displaystyle \frac{\|A\vector{r}_{j}\|_{2}}{\|A\vector{b}\|_{2}}$ vs. number of iterations for 
GMRES using pseudoinverse (blue), 
GMRES (red), 
and RRGMRES (green) for an inconsistent problem ({\bf ex32})}
\label{fig:Ar-pinv-gm-rr-32}
\end{minipage}
\begin{minipage}{0.5\hsize}
\begin{center}
\includegraphics[keepaspectratio,scale=0.4]{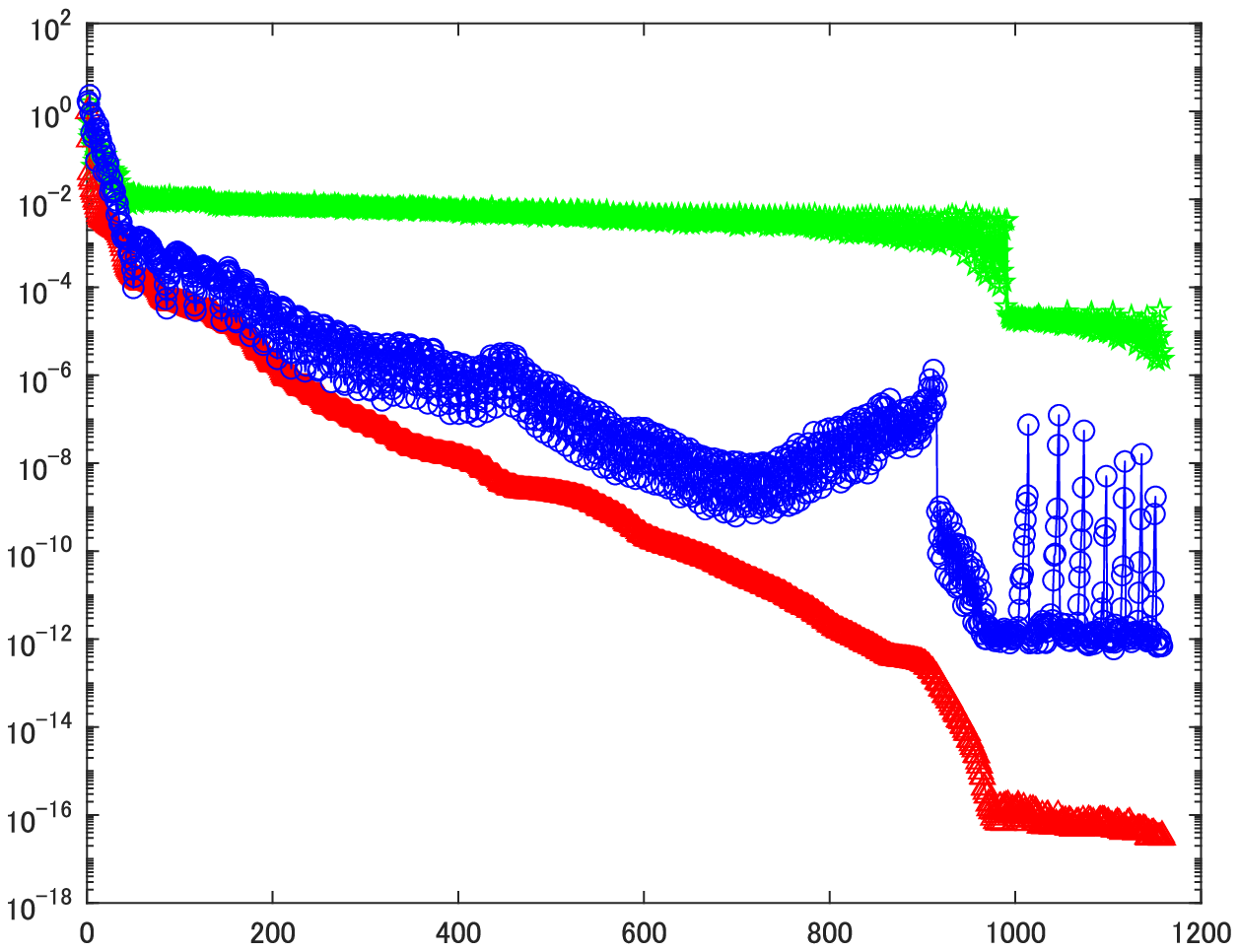}
\end{center}
\captionsetup{width=.95\linewidth}
\caption{$\displaystyle \frac{\|A\vector{r}_{j}\|_{2}}{\|A\vector{b}\|_{2}}$ (blue), 
$\frac{\sigma_{k}(H_{k+1,k})}{\sigma_{1}(H_{k+1,k})}$ (red) and 
$\frac{h_{j+1,j}}{\|H_{j,j}\|_{F}}$ (green) vs. number of iterations for 
GMRES using pseudoinverse for an inconsistent problem ({\bf ex32})}
\label{fig:Ar-sigma-h-pinv-32}
\end{minipage}
\end{figure}

\begin{figure}[htbp]
\begin{minipage}{0.5\hsize}
\begin{center}
\includegraphics[keepaspectratio,scale=0.4]{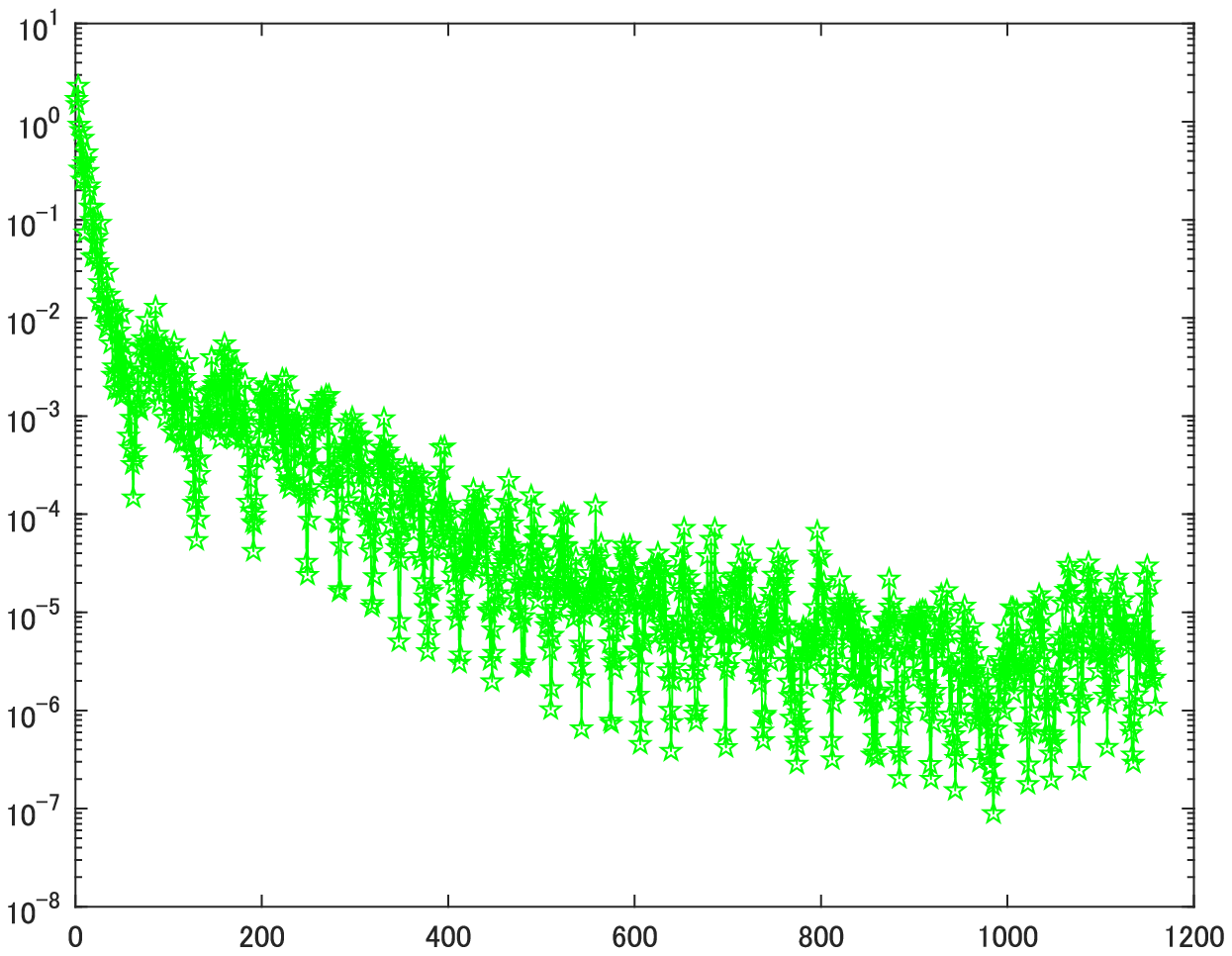}
\end{center}
\captionsetup{width=.95\linewidth}
\caption{$\displaystyle \frac{\|A\vector{r}_{j}\|_{2}}{\|A\vector{b}\|_{2}}$ vs. number of iterations for MINRES for an inconsistent problem ({\bf ex32})}
\label{fig:Ar-mi-32}
\end{minipage}
\begin{minipage}{0.5\hsize}
\begin{center}
\includegraphics[keepaspectratio,scale=0.4]{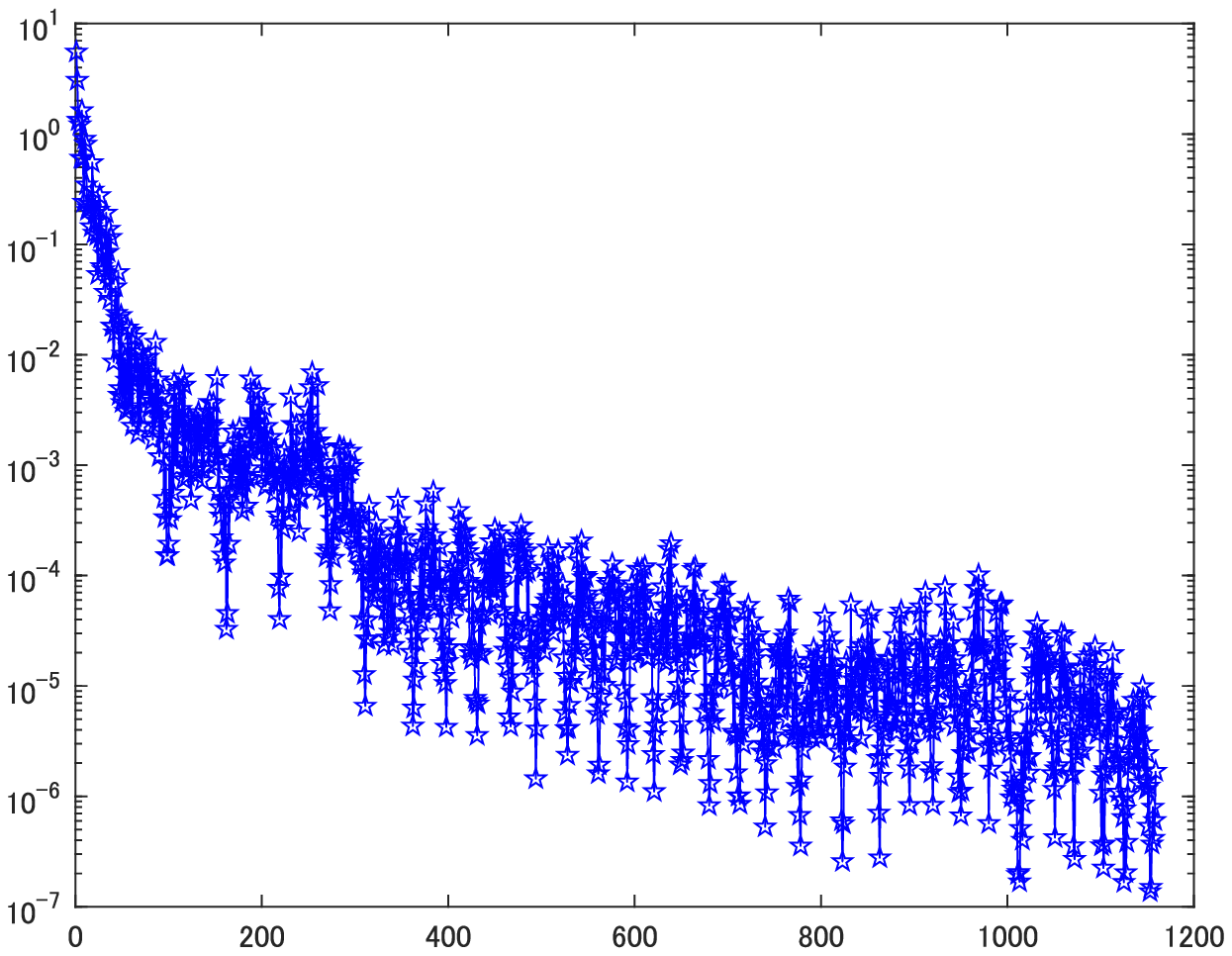}
\end{center}
\captionsetup{width=.95\linewidth}
\caption{$\displaystyle \frac{\|A\vector{r}_{j}\|_{2}}{\|A\vector{b}\|_{2}}$ vs. number of iterations for RRMINRES for an inconsistent problem ({\bf ex32})}
\label{fig:Ar-rrmi-32}
\end{minipage}
\end{figure}

We observe the following from Fig. \ref{fig:Ar-pinv-gm-rr-ms} and  Fig. \ref{fig:Ar-pinv-gm-rr-32}.
\begin{itemize}
\item The smallest value of $\displaystyle \frac{\|A\vector{r}_{j}\|_{2}}{\|A\vector{b}\|_{2}}$ of RRGMRES 
is smaller than the smallest values of $\displaystyle \frac{\|A\vector{r}_{j}\|_{2}}{\|A\vector{b}\|_{2}}$ of GMRES 
using pseudoinverse and GMRES.
\item $\displaystyle \frac{\|A\vector{r}_{j}\|_{2}}{\|A\vector{b}\|_{2}}$ of RRGMRES and GMRES diverges.
On the other hand, $\displaystyle \frac{\|A\vector{r}_{j}\|_{2}}{\|A\vector{b}\|_{2}}$ of GMRES using pseudoinverse 
does not diverge, although it oscillates.
\item $\displaystyle \frac{\|A\vector{r}_{j}\|_{2}}{\|A\vector{b}\|_{2}}$ of GMRES using 
pseudoinverse drastically decreases each time the smallest singular value of $H_{k+1,k}$ 
is truncated by pinv. 
\item $\displaystyle \frac{\|A\vector{r}_{j}\|_{2}}{\|A\vector{b}\|_{2}}$ of GMRES using 
pseudoinverse becomes smallest even when $\frac{h_{j+1,j}}{\|H_{j,j}\|_{F}}$ is not $O(\sqrt{u})$ 
(cf. Theorem \ref{Thm:StabGmres}). 
\end{itemize}

From Fig. \ref{fig:Ar-pinv-gm-rr-ms} and Fig. \ref{fig:Ar-rrmi-ms} for {\bf msc01050}, 
Fig. \ref{fig:Ar-pinv-gm-rr-32} and Fig. \ref{fig:Ar-rrmi-32} for {\bf ex32}, 
the smallest value of $\displaystyle \frac{\|A\vector{r}_{j}\|_{2}}{\|A\vector{b}\|_{2}}$ of RRGMRES is much smaller than
the smallest value of $\displaystyle \frac{\|A\vector{r}_{j}\|_{2}}{\|A\vector{b}\|_{2}}$ of RRMINRES for
inconsistent problems.
Thus, even for symmetric singular systems, RRGMRES and GMRES are better than RRMINRES and MINRES in finite
precision arithmetic.

\subsection{GMRES USING PSEUDOINVERSE FOR RANGE SYMMETRIC SYSTEMS}\label{sec:pinviryu}
Next, we will experiment with the following nonsymmetric but range symmetric system which arises from the 
finite difference discretization of a partial differential equation with periodic boundary condition as in \cite{Brown}.
\begin{eqnarray*}
\Delta u + d \frac{\partial u}{\partial x_{1}} = x_{1} + x_{2},~~x=(x_{1}, x_{2}) \in \Omega \equiv [0,1] \times [0,1]  \\
u(x_{1}, 0) = u(x_{1}, 1)~~and~~u(0, x_{2}) =  u(1,x_{2})~~for~~0 \leq x_{1}, x_{2} \leq 1 
\end{eqnarray*}
We discretized this boundary value problem with the usual second-order
centered differences on a $\rm{100} \times \rm{100}$ mesh with equally spaced discretization points, so that
the resulting linear systems are of dimension 10,000.
Assume that the matrix $A$ arises from this discretization. $A$ is normal and 
$ N(A) = N(A^{\rm T}) = R(A)^{\rm {\bot} } = {\rm span}{(1,1,...,1)}^{\rm T} $. Then, $A$ is range symmetric but nonsymmetric.
The right hand side vector $\vector{b}$ is a discretization of $x_{1}+x_{2}$. For $A$ and this $\vector{b}$,
(\ref{eqn:probset}) is inconsistent.
We apply GMRES using pseudoinverse to (\ref{eqn:probLS}).

Fig. \ref{fig:Ar-pinv-gm-rr-iryu} shows $\displaystyle \frac{\|A^{\rm T}\vector{r}_{j}\|_{2}}{\|A^{\rm T}\vector{b}\|_{2}}$ versus 
the iteration number for GMRES using pseudoinverse (blue), GMRES (red) and 
RRGMRES (green) for this inconsistent problem.
Fig. \ref{fig:Ar-sigma-h-pinv-iryu} show $\displaystyle \frac{\|A^{\rm T}\vector{r}_{j}\|_{2}}{\|A^{\rm T}\vector{b}\|_{2}}$  (blue),
$\frac{\sigma_{k}(H_{k+1,k})}{\sigma_{1}(H_{k+1,k})}$ (red) and 
$\frac{h_{j+1,j}}{\|H_{j,j}\|_{F}}$ (green) versus 
the iteration number for GMRES using pseudoinverse for this inconsistent problem.

\begin{figure}[htbp]
\begin{minipage}{0.5\hsize}
\begin{center}
\includegraphics[keepaspectratio,scale=0.4]{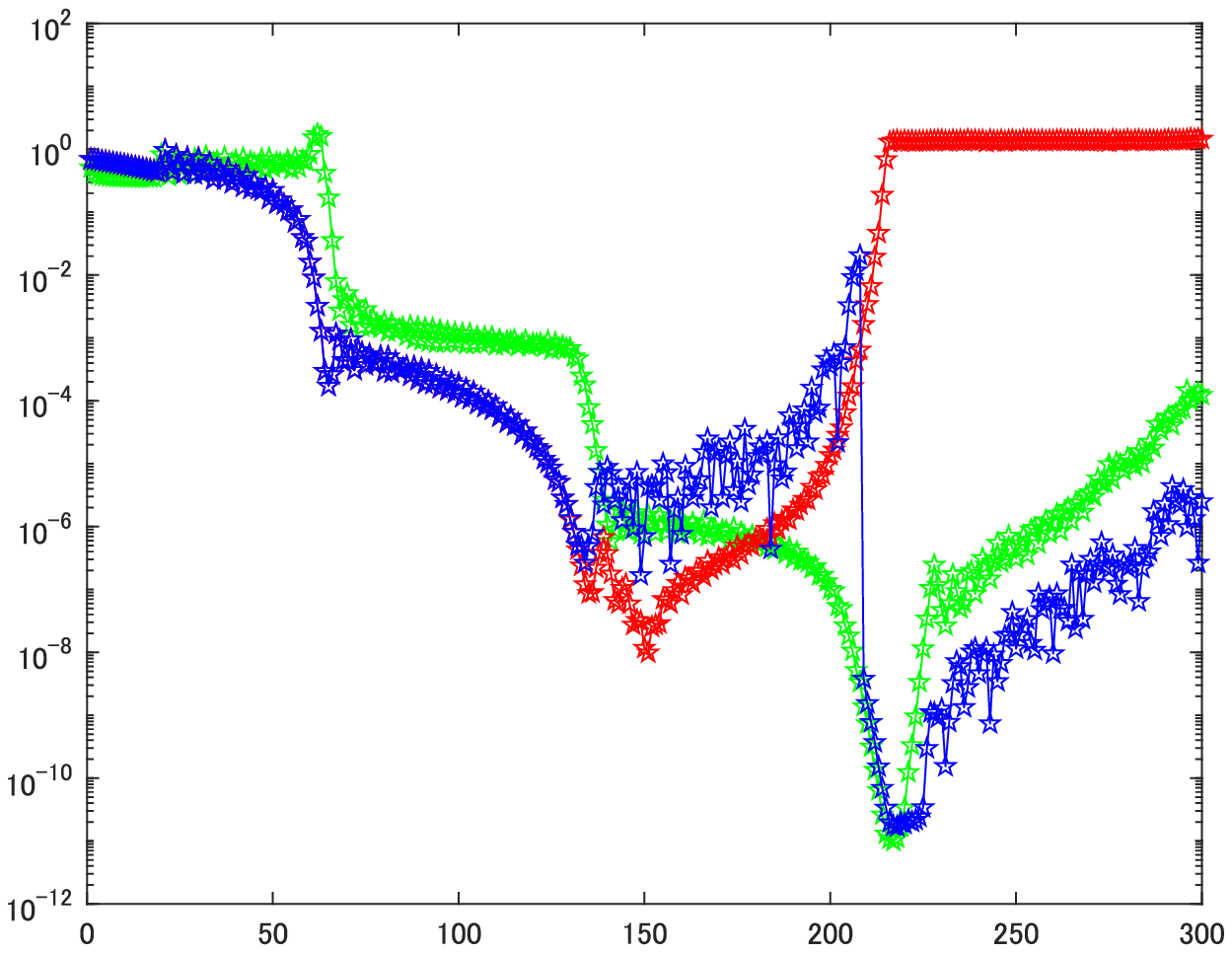}
\end{center}
\captionsetup{width=.95\linewidth}
\caption{$\displaystyle \frac{\|A^{\rm T}\vector{r}_{j}\|_{2}}{\|A^{\rm T}\vector{b}\|_{2}}$ vs. number of iterations for 
GMRES using pseudoinverse (blue), 
GMRES (red), 
and RRGMRES (green) for a range symmetric inconsistent problem}
\label{fig:Ar-pinv-gm-rr-iryu}
\end{minipage}
\begin{minipage}{0.5\hsize}
\begin{center}
\includegraphics[keepaspectratio,scale=0.4]{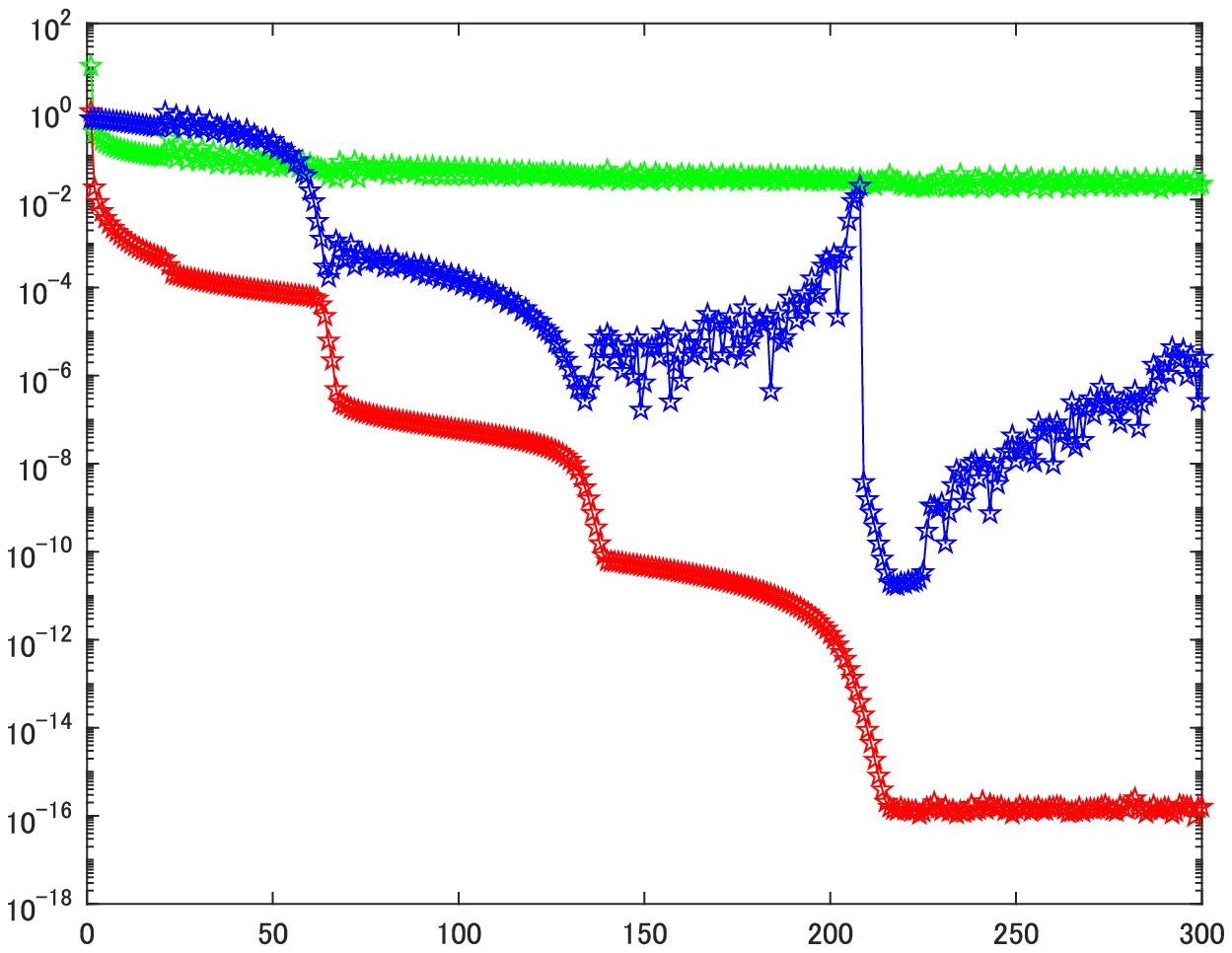}
\end{center}
\captionsetup{width=.95\linewidth}
\caption{$\displaystyle \frac{\|A^{\rm T}\vector{r}_{j}\|_{2}}{\|A^{\rm T}\vector{b}\|_{2}}$ (blue), 
$\frac{\sigma_{k}(H_{k+1,k})}{\sigma_{1}(H_{k+1,k})}$ (red) and 
$\frac{h_{j+1,j}}{\|H_{j,j}\|_{F}}$ (green) vs. number of iterations for 
GMRES using pseudoinverse for a range symmetric inconsistent problem}
\label{fig:Ar-sigma-h-pinv-iryu}
\end{minipage}
\end{figure}

We observe the following from Fig. \ref{fig:Ar-pinv-gm-rr-iryu}.
\begin{itemize}
\item The smallest value of $\displaystyle \frac{\|A^{\rm T}\vector{r}_{j}\|_{2}}{\|A^{\rm T}\vector{b}\|_{2}}$ of RRGMRES 
is smaller than the smallest values of $\displaystyle \frac{\|A^{\rm T}\vector{r}_{j}\|_{2}}{\|A^{\rm T}\vector{b}\|_{2}}$ of GMRES 
using pseudoinverse and GMRES.
\item  $\displaystyle \frac{\|A^{\rm T}\vector{r}_{j}\|_{2}}{\|A^{\rm T}\vector{b}\|_{2}}$ of GMRES diverges. After 221 iteration steps, 
$\displaystyle \frac{\|A^{\rm T}\vector{r}_{j}\|_{2}}{\|A^{\rm T}\vector{b}\|_{2}}$ of GMRES using pseudoinverse is smaller 
than $\displaystyle \frac{\|A^{\rm T}\vector{r}_{j}\|_{2}}{\|A^{\rm T}\vector{b}\|_{2}}$ of RRGMRES.
\item $\displaystyle \frac{\|A^{\rm T}\vector{r}_{j}\|_{2}}{\|A^{\rm T}\vector{b}\|_{2}}$ of GMRES using pseudoinverse increases after 
$\displaystyle \frac{\|A^{\rm T}\vector{r}_{j}\|_{2}}{\|A^{\rm T}\vector{b}\|_{2}}$ of this method becomes smallest.
\item  $\displaystyle \frac{\|A^{\rm T}\vector{r}_{j}\|_{2}}{\|A^{\rm T}\vector{b}\|_{2}}$ of GMRES using 
pseudoinverse becomes smallest 
even when 
$\frac{h_{j+1,j}}{\|H_{j,j}\|_{F}} > O(\sqrt{u})$
(cf. Theorem \ref{Thm:StabGmres}).
\end{itemize}

In the next section, we will further improve the convergence of GMRES using pseudoinverse by 
using reorthogonalization of the Arnoldi process to suppress the oscillation and the increasing of the residual norm.

\section{NUMERICAL EXPERIMENT ON GMRES USING PSEUDOINVERSE AND REORTHOGONALIZATION}\label{sec:StabReorth}
We think that  $\displaystyle \frac{\|A^{\rm T}\vector{r}_{j}\|_{2}}{\|A^{\rm T}\vector{b}\|_{2}}$ of 
GMRES using pseudoinverse oscillates because the column vectors of $V_{k}$ become linearly dependent.
Thus, we think that we can remove the oscillation of $\displaystyle \frac{\|A^{\rm T}\vector{r}_{j}\|_{2}}{\|A^{\rm T}\vector{b}\|_{2}}$ of 
GMRES using pseudoinverse by keeping the linear independence of the column vectors of $V_{k}$ by
reorthogonalization, as proposed in \cite{Liao}.

The algorithm of the reorthogonalization part in the Modified Gram-Schmidt with reorthogonalization is as follows.

\vspace{6pt}
\hspace{-13pt}{\bf Algorithm 4 : Reorthogonalization part of the Modified Gram-Schmidt with reorthogonalization}
\vspace{4pt}
\\1:~~$h_{i,j} = (\vector{v}_{i}, \vector{v}_{j})~(i=1,2,...j)$ 
\\2:~~$\vector{w}=A\vector{v}_{j} - {{\sum^{j}}_{i=1}}h_{i,j}\vector{v}_{i}$
\\3:~~$\hat{\vector{v}}_{j+1} = \vector{w} - {{\sum^{j}}_{i=1}}(\vector{w}, \vector{v}_{i})$ 
\\4:~~$h_{j+1,j}=\|\hat{\vector{v}}_{j+1}\|_{2}$
\\5:~~If $h_{j+1,j} \ne 0$, $\vector{v}_{j+1} = \frac{\hat{\vector{v}}_{j+1}}{h_{j+1,j}}$
\vspace{4pt}

In {\bf Algorithm 4}, line 3 is the reorthogonalization part.
For the same inconsistent systems in the previous section,
we will report the numerical results on GMRES using pseudoinverse and reorthogonalization.

Here, let $\sigma_{k-1}(H_{k+1,k})$ be the 2nd smallest singular value of $H_{k+1,k}$, 
$\sigma_{k-2}(H_{k+1,k})$ be the 3rd smallest singular value of $H_{k+1,k}$ 
and $\sigma_{k-3}(H_{k+1,k})$ be the 4th smallest singular value of $H_{k+1,k}$. 

\subsection{GMRES USING PSEUDOINVERSE AND REORTHOGONALIZATION FOR SYMMETRIC MATRICES}\label{sec:preinvsym}

Fig. \ref{fig:Ar-pinv-re2-gm-re2-rr-ms} for {\bf msc01050} 
and
Fig. \ref{fig:Ar-pinv-re2-gm-re2-rr-32} for {\bf ex32}
show $\displaystyle \frac{\|A\vector{r}_{j}\|_{2}}{\|A\vector{b}\|_{2}}$ versus 
the iteration number for GMRES using pseudoinverse and reorthogonalization (blue),
GMRES using reorthogonalization (red) and 
RRGMRES (green) for an inconsistent problem.

Fig. \ref{fig:Ar-sigma-h-pinv-re2-ms} for {\bf msc01050}
and
Fig. \ref{fig:Ar-sigma-h-pinv-re2-32} for {\bf ex32}  
show $\displaystyle \frac{\|A\vector{r}_{j}\|_{2}}{\|A\vector{b}\|_{2}}$ 
(blue), $\frac{\sigma_{k}(H_{k+1,k})}{\sigma_{1}(H_{k+1,k})}$ (red), and $\frac{h_{k+1,k}}{\|H_{k,k}\|_{F}}$ (green) 
versus 
the iteration number for GMRES using pseudoinverse and reorthogonalization 
for an inconsistent problem.

Fig. \ref{fig:Ar-pinv-re2-gm-re2-rr-ms}
and
Fig. \ref{fig:Ar-pinv-re2-gm-re2-rr-32} show that 
the reorthogonalization eliminates the oscillation of GMRES using pseudoinverse.


Fig. \ref{fig:sig-min1234-gm-ms} and Fig. \ref{fig:sig-min1234-gm-re2-ms} for {\bf msc01050},
Fig. \ref{fig:sig-min1234-gm-32} and Fig. \ref{fig:sig-min1234-gm-re2-32} for {\bf ex32} 
show $ \frac{ \sigma_{k} ( H_{k+1,k} ) }{ \sigma_1 (H_{k+1,k} ) } $, 
$ \frac{ \sigma_{k-1} ( H_{k+1,k} ) }{ \sigma_1 ( H_{k+1,k} ) } $,
$\frac{\sigma_{k-2}(H_{k+1,k})}{\sigma_1(H_{k+1,k})}$, 
$\frac{\sigma_{k-3}(H_{k+1,k})}{\sigma_1(H_{k+1,k})}$
of 
GMRES using pseudoinverse, 
and GMRES using pseudoinverse and reorthogonalization.


\begin{figure}[htbp]
\begin{minipage}{0.5\hsize}
\begin{center}
\includegraphics[keepaspectratio,scale=0.4]{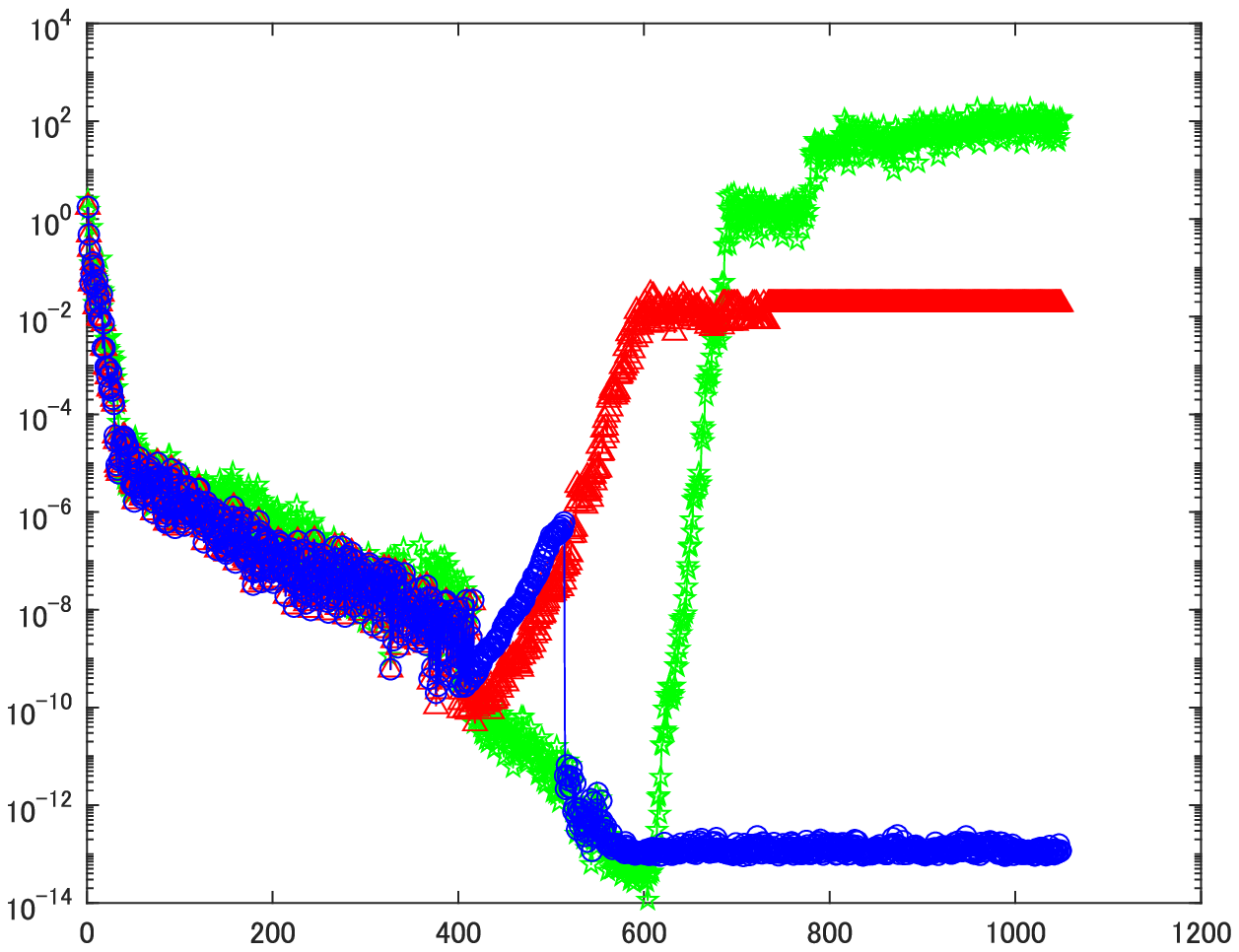}
\end{center}
\captionsetup{width=.95\linewidth}
\caption{$\displaystyle \frac{\|A\vector{r}_{j}\|_{2}}{\|A\vector{b}\|_{2}}$ vs. number of iterations for 
GMRES using pseudoinverse and reorthogonalization (blue), 
GMRES using reorthogonalization (red), 
and RRGMRES (green) for an inconsistent problem ({\bf msc01050})}
\label{fig:Ar-pinv-re2-gm-re2-rr-ms}
\end{minipage}
\begin{minipage}{0.5\hsize}
\begin{center}
\includegraphics[keepaspectratio,scale=0.4]{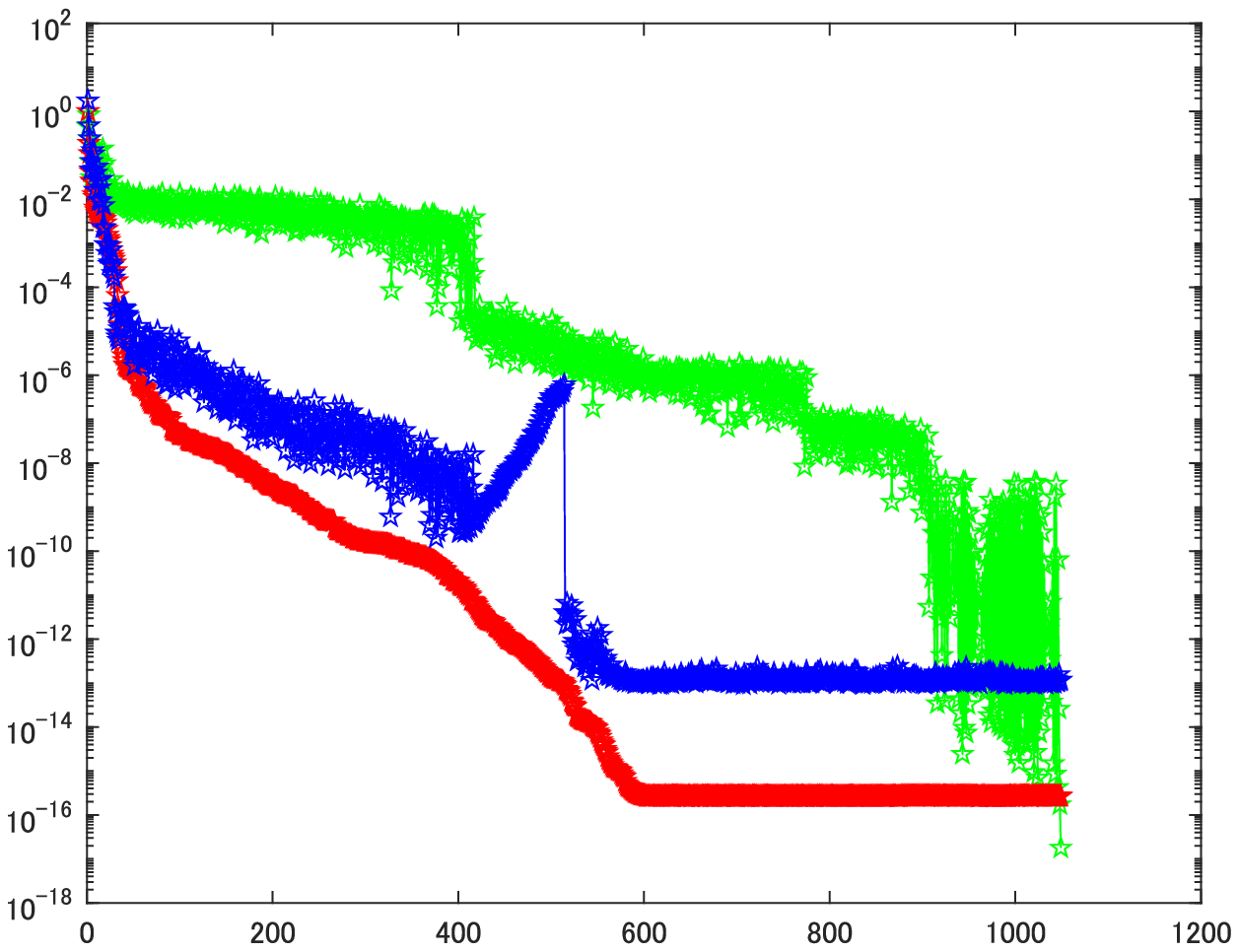}
\end{center}
\captionsetup{width=.95\linewidth}
\caption{$\displaystyle \frac{\|A\vector{r}_{j}\|_{2}}{\|A\vector{b}\|_{2}}$ (blue), 
$\frac{\sigma_{k}(H_{k+1,k})}{\sigma_{1}(H_{k+1,k})}$ (red) and 
$\frac{h_{j+1,j}}{\|H_{j,j}\|_{F}}$ (green) vs. number of iterations for 
GMRES using pseudoinverse and reorthogonalization for an inconsistent problem ({\bf msc01050})}
\label{fig:Ar-sigma-h-pinv-re2-ms}
\end{minipage}
\end{figure}

\begin{figure}[htbp]
\begin{minipage}{0.5\hsize}
\begin{center}
\includegraphics[keepaspectratio,scale=0.4]{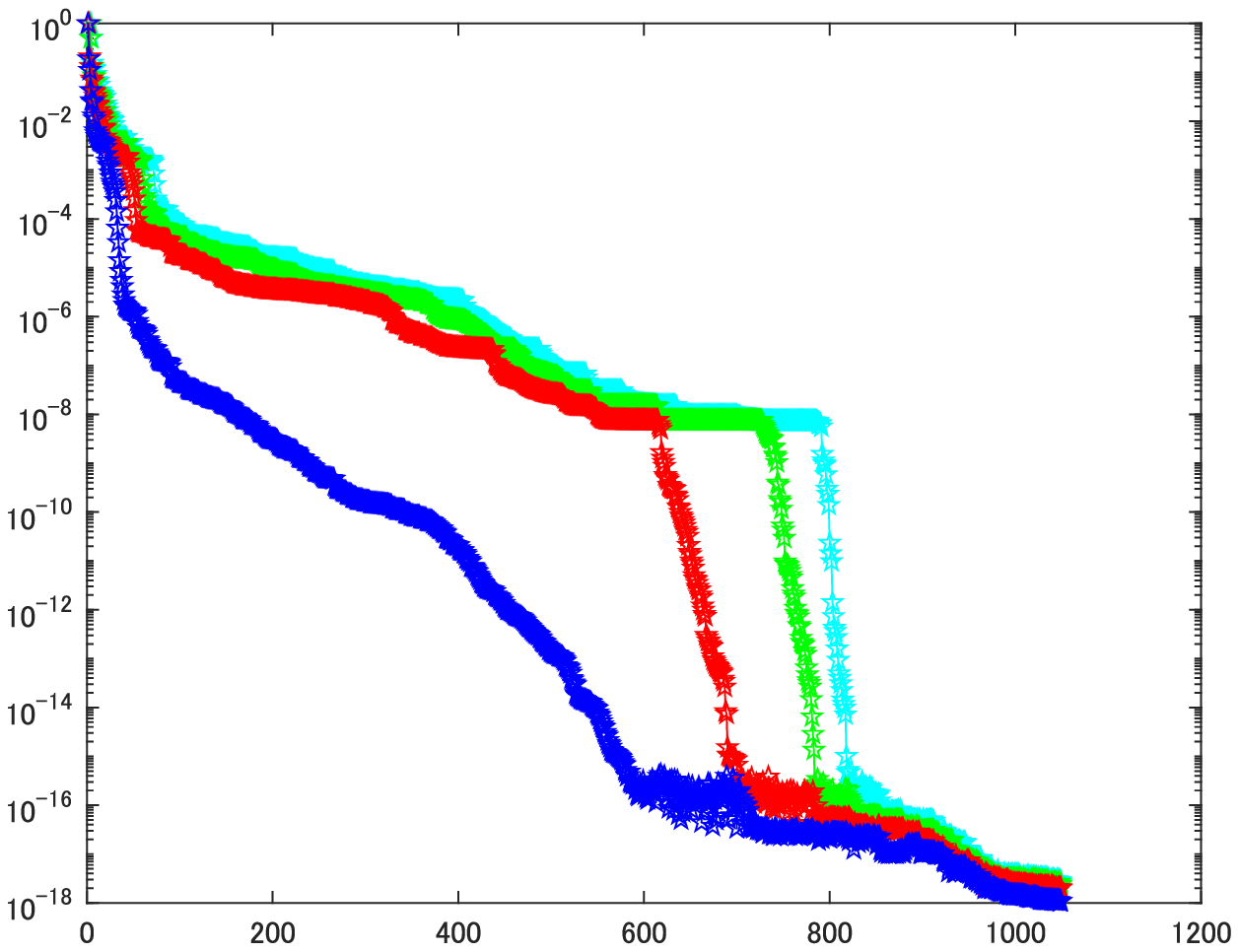}
\end{center}
\captionsetup{width=.95\linewidth}
%
\caption{$\frac{\sigma_{k}(H_{k+1,k})}{\sigma_{1}(H_{k+1,k})}$(blue), 
$\frac{\sigma_{k-1}(H_{k+1,k})}{\sigma_{1}(H_{k+1,k})}$(red),
$\frac{\sigma_{k-2}(H_{k+1,k})}{\sigma_{1}(H_{k+1,k})}$(green) and 
$\frac{\sigma_{k-3}(H_{k+1,k})}{\sigma_{1}(H_{k+1,k})}$(cian) \\
vs. number of iterations for GMRES using pseudoinverse 
for an inconsistent problem ({\bf msc01050})}
\label{fig:sig-min1234-gm-ms}
\end{minipage}
\begin{minipage}{0.5\hsize}
\begin{center}
\includegraphics[keepaspectratio,scale=0.4]{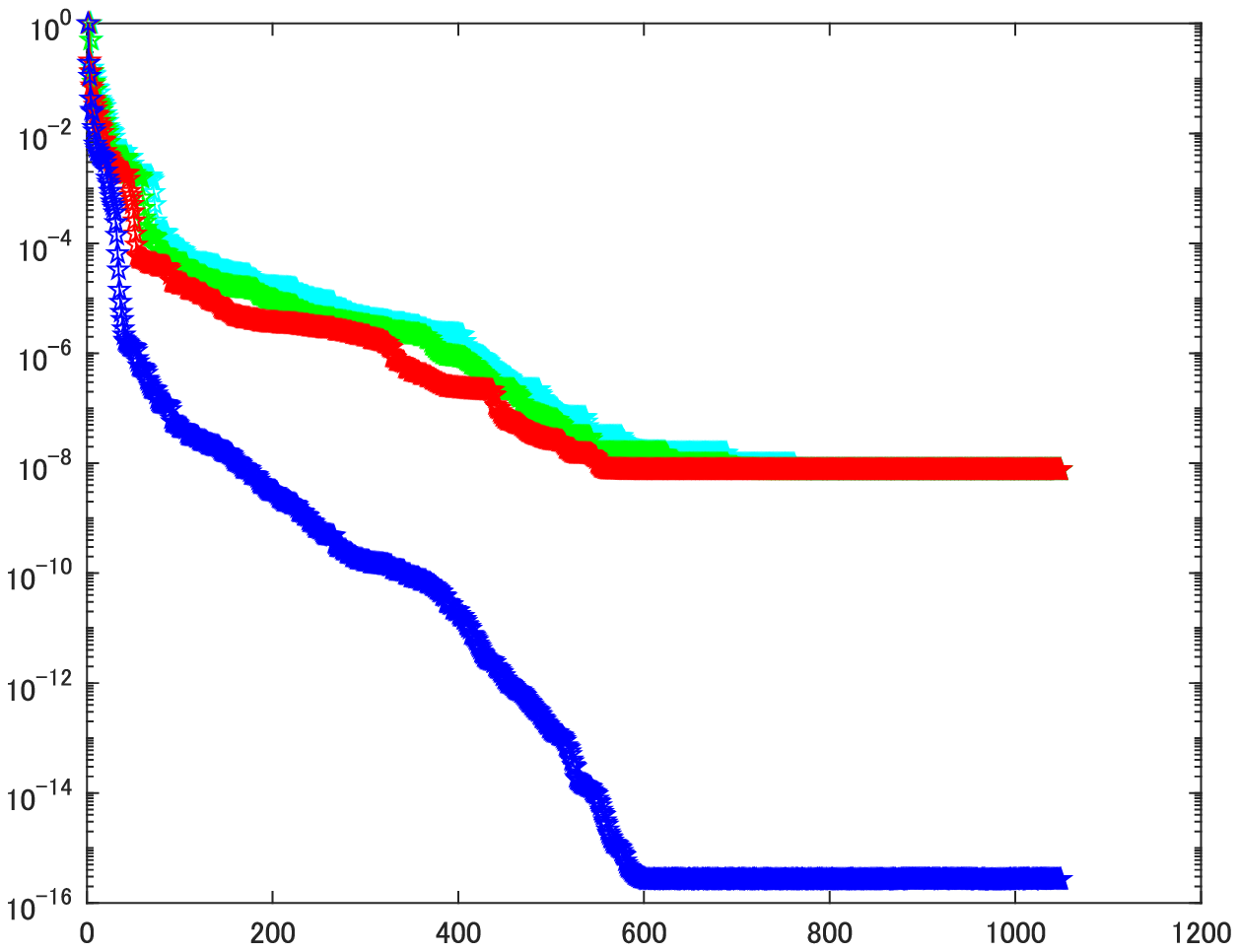}
\end{center}
\captionsetup{width=.95\linewidth}
%
\caption{$\frac{\sigma_{k}(H_{k+1,k})}{\sigma_{1}(H_{k+1,k})}$(blue), 
$\frac{\sigma_{k-1}(H_{k+1,k})}{\sigma_{1}(H_{k+1,k})}$(red),
$\frac{\sigma_{k-2}(H_{k+1,k})}{\sigma_{1}(H_{k+1,k})}$(green) and 
$\frac{\sigma_{k-3}(H_{k+1,k})}{\sigma_{1}(H_{k+1,k})}$(cian) \\ 
vs. number of iterations for GMRES using pseudoinverse and reorthogonalization
for an inconsistent problem ({\bf msc01050})}
\label{fig:sig-min1234-gm-re2-ms}
\end{minipage}
\end{figure}

\begin{figure}[htbp]
\begin{minipage}{0.5\hsize}
\begin{center}
\includegraphics[keepaspectratio,scale=0.4]{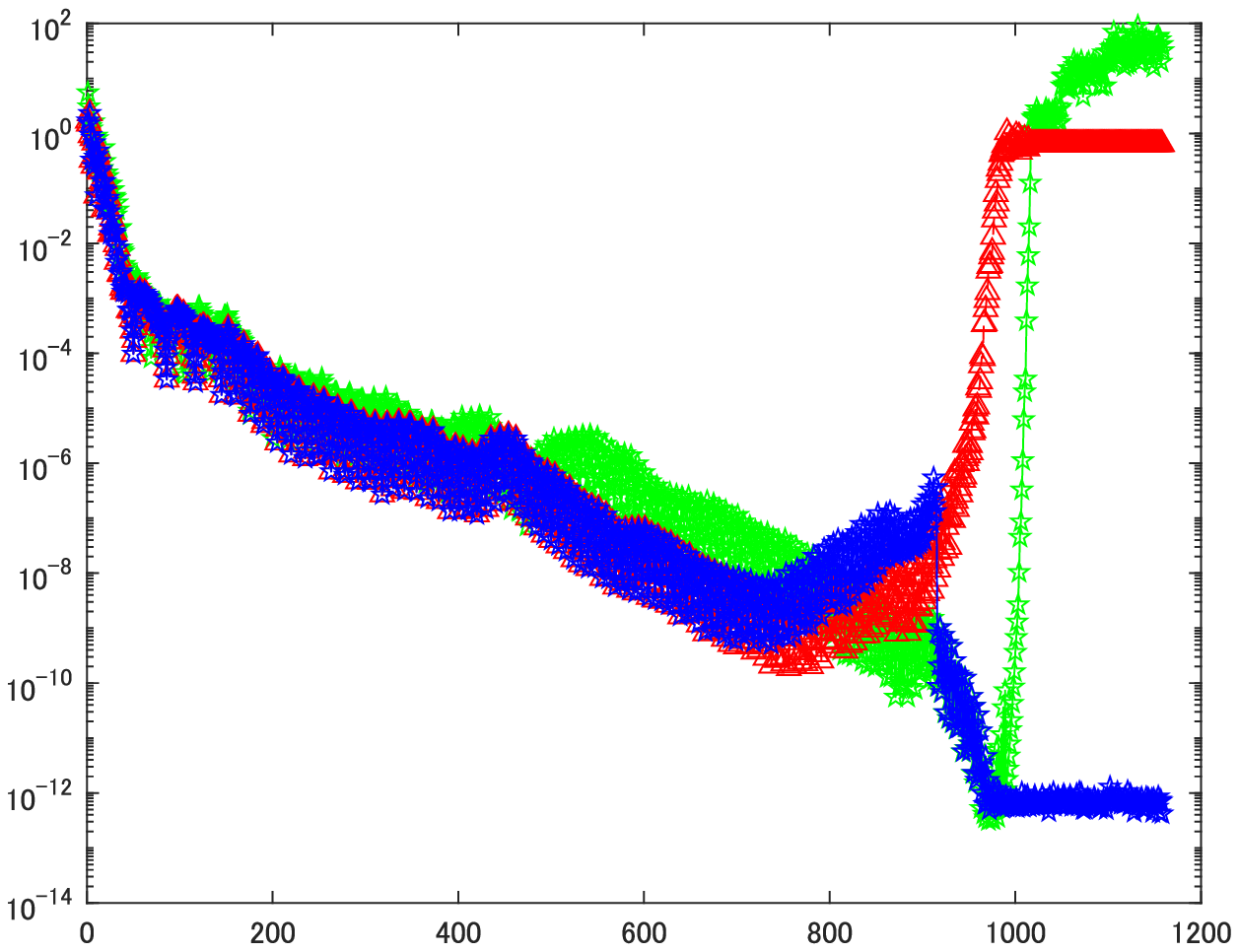}
\end{center}
\captionsetup{width=.95\linewidth}
\caption{$\displaystyle \frac{\|A\vector{r}_{j}\|_{2}}{\|A\vector{b}\|_{2}}$ vs. number of iterations for 
GMRES using pseudoinverse and reorthogonalization (blue), 
GMRES using reorthogonalization (red), 
and RRGMRES (green) for an inconsistent problem ({\bf ex32})}
\label{fig:Ar-pinv-re2-gm-re2-rr-32}
\end{minipage}
\begin{minipage}{0.5\hsize}
\begin{center}
\includegraphics[keepaspectratio,scale=0.4]{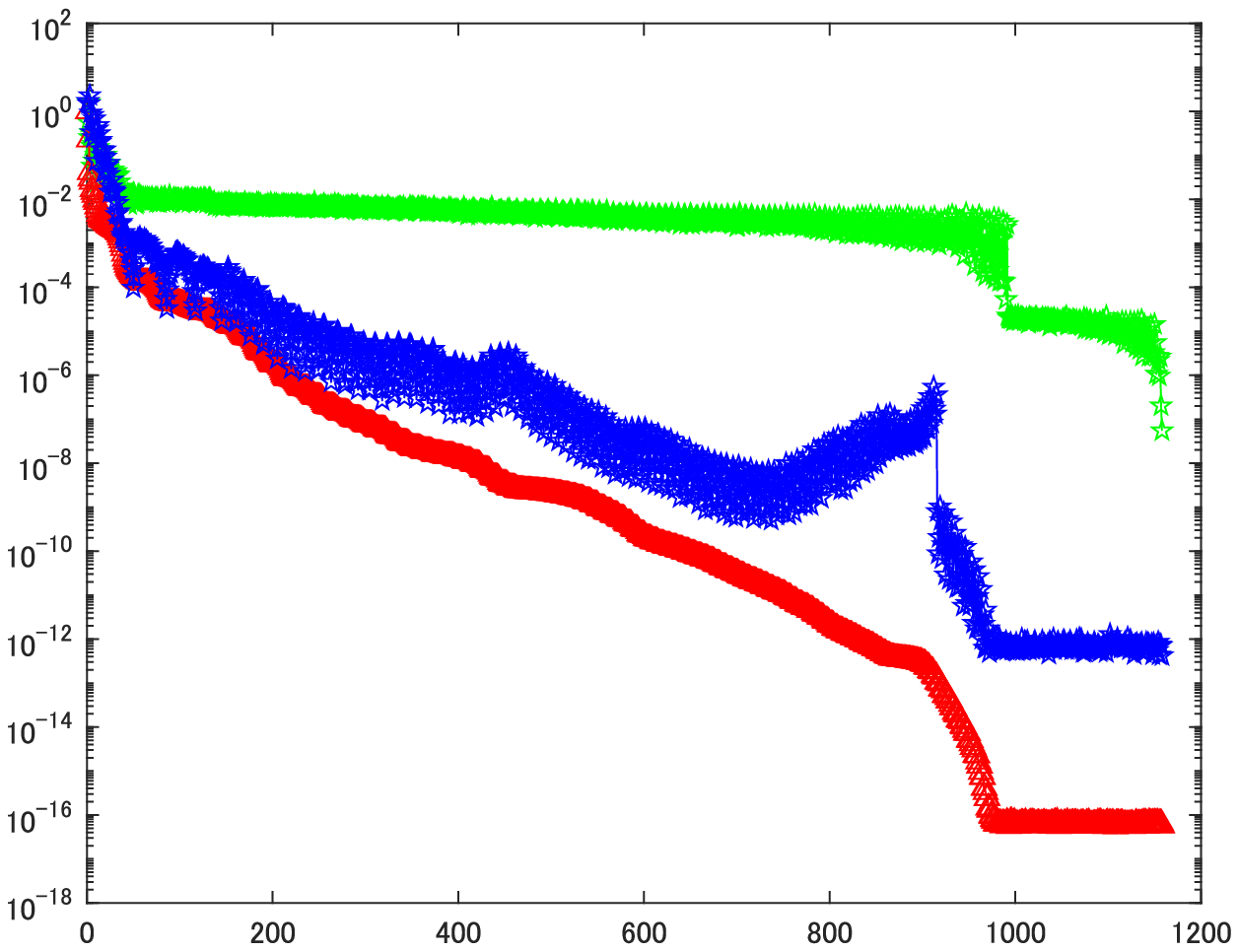}
\end{center}
\captionsetup{width=.95\linewidth}
\caption{$\displaystyle \frac{\|A\vector{r}_{j}\|_{2}}{\|A\vector{b}\|_{2}}$ (blue), 
$\frac{\sigma_{k}(H_{k+1,k})}{\sigma_{1}(H_{k+1,k})}$ (red) and 
$\frac{h_{j+1,j}}{\|H_{j,j}\|_{F}}$ (green) vs. number of iterations for 
GMRES using pseudoinverse and reorthogonalization for an inconsistent problem ({\bf ex32})}
\label{fig:Ar-sigma-h-pinv-re2-32}
\end{minipage}
\end{figure}

\begin{figure}[htbp]
\begin{minipage}{0.5\hsize}
\begin{center}
\includegraphics[keepaspectratio,scale=0.4]{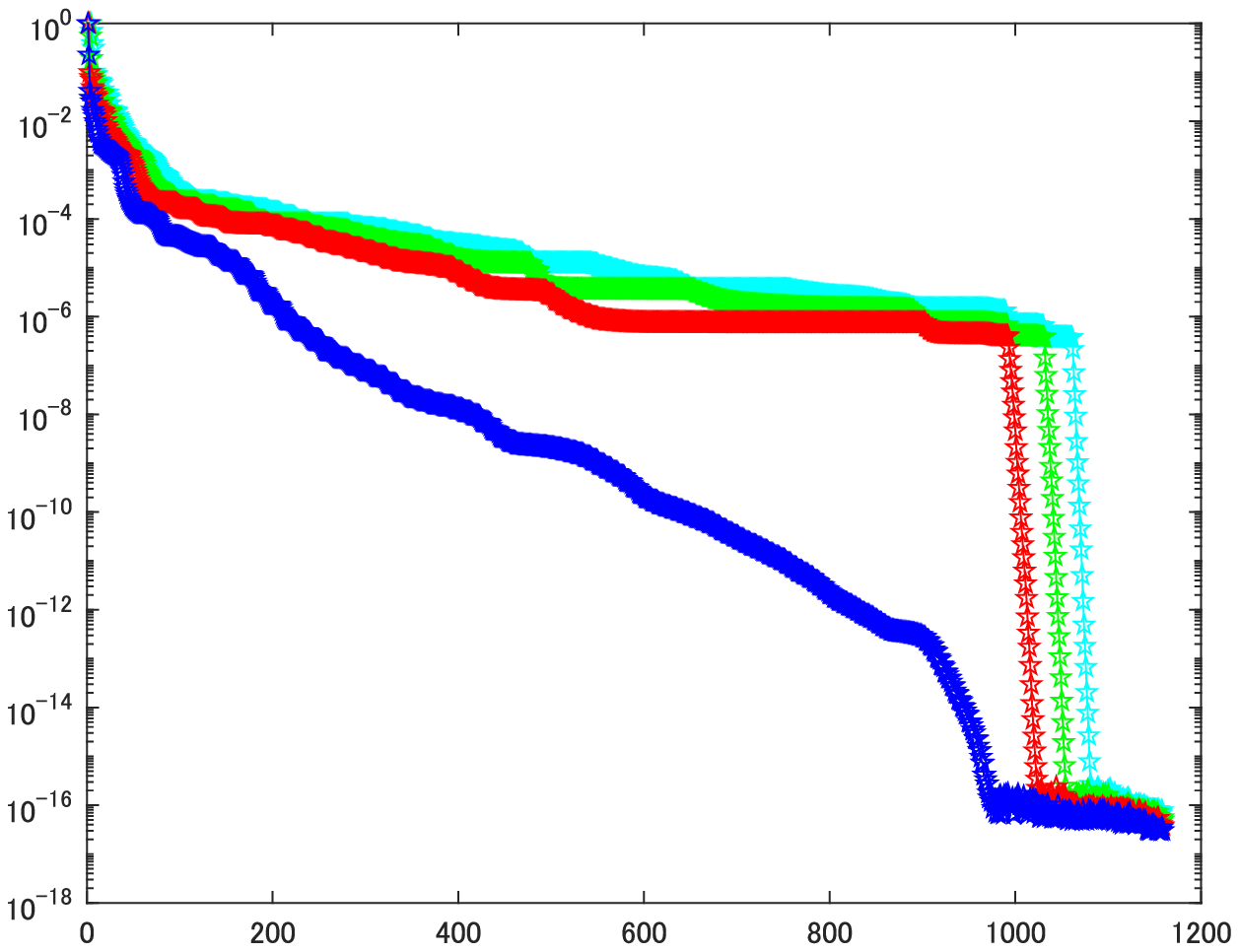}
\end{center}
\captionsetup{width=.95\linewidth}
%
\caption{$\frac{\sigma_{k}(H_{k+1,k})}{\sigma_{1}(H_{k+1,k})}$(blue), 
$\frac{\sigma_{k-1}(H_{k+1,k})}{\sigma_{1}(H_{k+1,k})}$(red),
$\frac{\sigma_{k-2}(H_{k+1,k})}{\sigma_{1}(H_{k+1,k})}$(green) and 
$\frac{\sigma_{k-3}(H_{k+1,k})}{\sigma_{1}(H_{k+1,k})}$(cian) \\
vs. number of iterations for GMRES using pseudoinverse 
for an inconsistent problem ({\bf ex32})}
\label{fig:sig-min1234-gm-32}
\end{minipage}
\begin{minipage}{0.5\hsize}
\begin{center}
\includegraphics[keepaspectratio,scale=0.4]{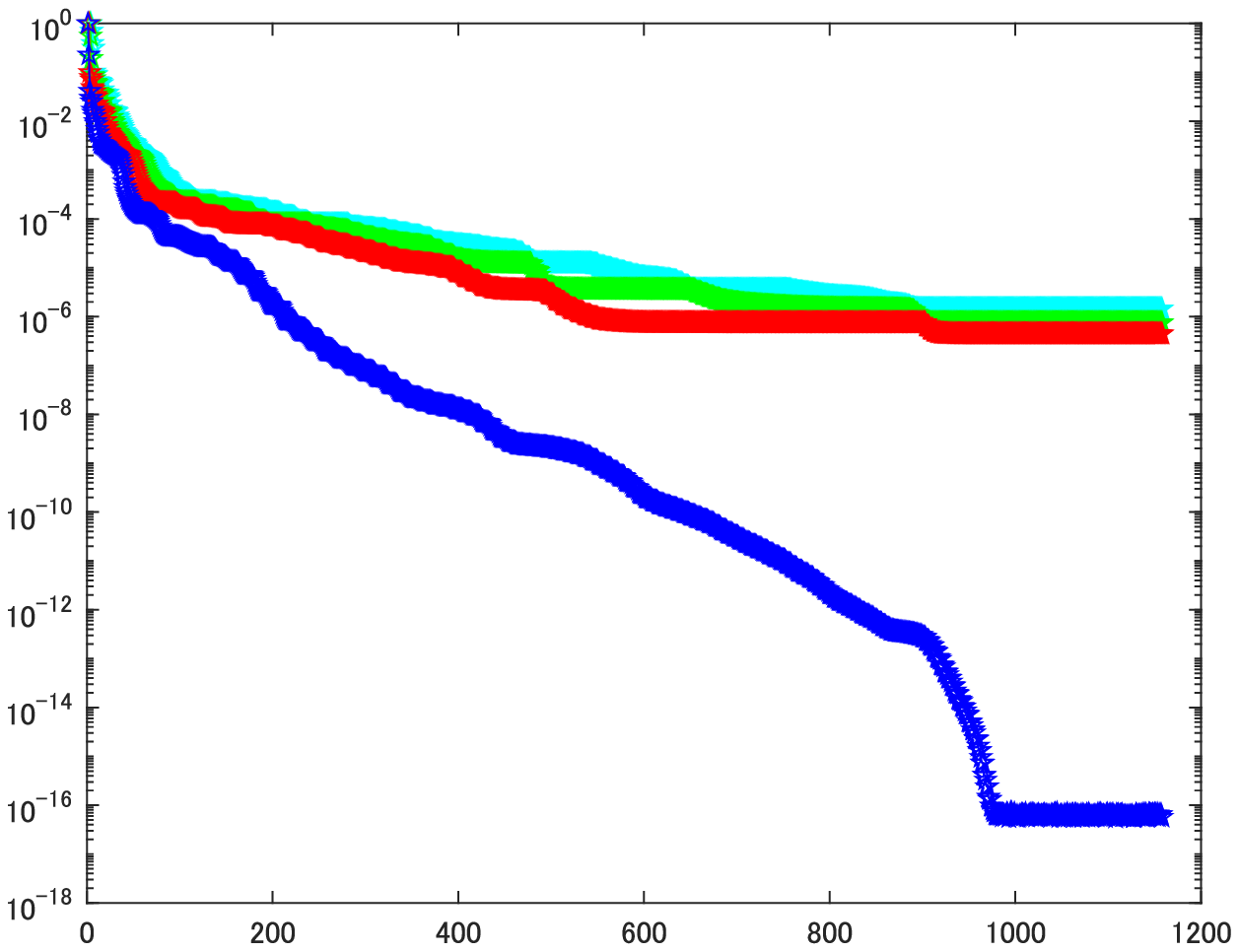}
\end{center}
\captionsetup{width=.95\linewidth}
%
\caption{$\frac{\sigma_{k}(H_{k+1,k})}{\sigma_{1}(H_{k+1,k})}$(blue), 
$\frac{\sigma_{k-1}(H_{k+1,k})}{\sigma_{1}(H_{k+1,k})}$(red),
$\frac{\sigma_{k-2}(H_{k+1,k})}{\sigma_{1}(H_{k+1,k})}$(green) and 
$\frac{\sigma_{k-3}(H_{k+1,k})}{\sigma_{1}(H_{k+1,k})}$(cian) \\
vs. number of iterations for GMRES using pseudoinverse and reorthogonalization
for an inconsistent problem ({\bf ex32})}
\label{fig:sig-min1234-gm-re2-32}
\end{minipage}
\end{figure}


Fig. \ref{fig:sig-min1234-gm-ms}
and
Fig. \ref{fig:sig-min1234-gm-32}
show that 
$\frac{\sigma_{k}(H_{k+1,k})}{\sigma_{1}(H_{k+1,k})}$,
$\frac{\sigma_{k-1}(H_{k+1,k})}{\sigma_{1}(H_{k+1,k})}$,
$\frac{\sigma_{k-2}(H_{k+1,k})}{\sigma_{1}(H_{k+1,k})}$ and 
$\frac{\sigma_{k-3}(H_{k+1,k})}{\sigma_{1}(H_{k+1,k})}$ of 
GMRES cluster as the GMRES iterations proceed. 
For example, 
$\sigma_{k-1}(H_{k+1,k})$ 
is initially larger than $tol$ 
and is not truncated, but gradually, it decreases, and when it becomes
smaller than $tol$, it is truncated by using pseudoinverse.
Similarly for
$\sigma_{k-2}(H_{k+1,k})$ and 
$\sigma_{k-3}(H_{k+1,k})$.  
This is why $\displaystyle \frac{\|A\vector{r}\|_{2}}{\|A\vector{b}\|_{2}}$ of GMRES using pseudoinverse without reorthogonalization 
oscillates in Fig. \ref{fig:Ar-pinv-gm-rr-ms}
and Fig. \ref{fig:Ar-pinv-gm-rr-32}.

On the other hand, Fig. \ref{fig:sig-min1234-gm-re2-ms} and Fig. \ref{fig:sig-min1234-gm-re2-32} show 
that 
$\frac{\sigma_{k-1}(H_{k+1,k})}{\sigma_{1}(H_{k+1,k})}$,
$\frac{\sigma_{k-2}(H_{k+1,k})}{\sigma_{1}(H_{k+1,k})}$ and 
$\frac{\sigma_{k-3}(H_{k+1,k})}{\sigma_{1}(H_{k+1,k})}$ of
GMRES 
using pseudoinverse and reorthogonalization
are larger than $10^{-9}$ 
even when the iterations proceed,
whereas 
$\frac{\sigma_{k}(H_{k+1,k})}{\sigma_{1}(H_{k+1,k})}$ 
becomes smaller than $10^{-15}$.

Since 
$\frac{\sigma_{k}(H_{k+1,k})}{\sigma_{1}(H_{k+1,k})}$,
$\frac{\sigma_{k-1}(H_{k+1,k})}{\sigma_{1}(H_{k+1,k})}$,
$\frac{\sigma_{k-2}(H_{k+1,k})}{\sigma_{1}(H_{k+1,k})}$ and 
$\frac{\sigma_{k-3}(H_{k+1,k})}{\sigma_{1}(H_{k+1,k})}$
of {\bf msc01050}
and
{\bf ex32} do not cluster
and the column vectors of $V_{k}$ are kept numerically linearly independent by reorthogonalization,
all of
$\frac{\sigma_{k}(H_{k+1,k})}{\sigma_{1}(H_{k+1,k})}$,
$\frac{\sigma_{k-1}(H_{k+1,k})}{\sigma_{1}(H_{k+1,k})}$,
$\frac{\sigma_{k-2}(H_{k+1,k})}{\sigma_{1}(H_{k+1,k})}$ and 
$\frac{\sigma_{k-3}(H_{k+1,k})}{\sigma_{1}(H_{k+1,k})}$ 
do not cluster. Thus, $\displaystyle \frac{\|A\vector{r}\|_{2}}{\|A\vector{b}\|_{2}}$
of GMRES using pseudoinverse and reorthogonalization 
does not oscillate.

\subsection{GMRES USING PSEUDOINVERSE AND REORTHOGONALIZATION FOR RANGE SYMMETRIC SYSTEMS}\label{sec:preinviryu}
We will experiment with the nonsymmetric but range symmetric system in section \ref{sec:pinviryu}.

Fig. \ref{fig:Ar-pinv-re2-gm-re2-rr-iryu} shows $\displaystyle \frac{\|A^{\rm T}\vector{r}_{j}\|_{2}}{\|A^{\rm T}\vector{b}\|_{2}}$ versus 
the iteration number for GMRES using pseudoinverse and reorthogonalization (blue),
GMRES using reorthogonalization (red) and 
RRGMRES (green) for an inconsistent problem.

Fig. \ref{fig:Ar-sigma-h-pinv-re2-iryu} shows $\displaystyle \frac{\|A^{\rm T}\vector{r}_{j}\|_{2}}{\|A^{\rm T}\vector{b}\|_{2}}$ 
(blue), $\frac{\sigma_{k}(H_{k+1,k})}{\sigma_{1}(H_{k+1,k})}$ (red), and $\frac{h_{k+1,k}}{\|H_{k,k}\|_{F}}$ (green) 
versus 
the iteration number for GMRES using pseudoinverse and reorthogonalization 
for an inconsistent problem.

Fig. \ref{fig:sig-min1234-gm-iryu} and Fig. \ref{fig:sig-min1234-gm-re2-iryu}
show 
$\frac{\sigma_{k}(H_{k+1,k})}{\sigma_{1}(H_{k+1,k})}$,
$\frac{\sigma_{k-1}(H_{k+1,k})}{\sigma_{1}(H_{k+1,k})}$,
$\frac{\sigma_{k-2}(H_{k+1,k})}{\sigma_{1}(H_{k+1,k})}$ and 
$\frac{\sigma_{k-3}(H_{k+1,k})}{\sigma_{1}(H_{k+1,k})}$ of
GMRES using pseudoinverse,
and GMRES using pseudoinverse and reorthogonalization.

\begin{figure}[htbp]
\begin{minipage}{0.5\hsize}
\begin{center}
\includegraphics[keepaspectratio,scale=0.4]{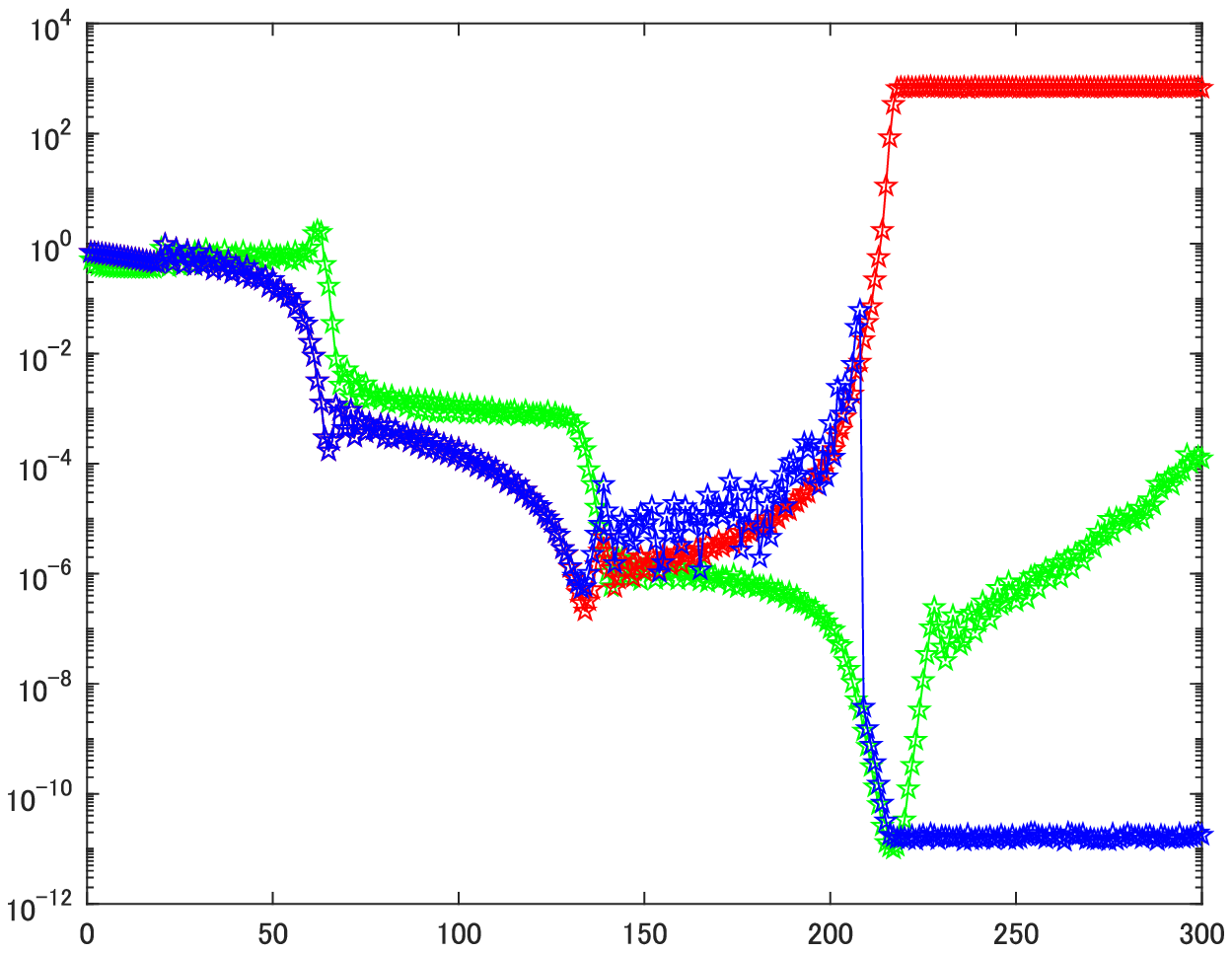}
\end{center}
\captionsetup{width=.95\linewidth}
\caption{$\displaystyle \frac{\|A^{\rm T}\vector{r}_{j}\|_{2}}{\|A^{\rm T}\vector{b}\|_{2}}$ vs. number of iterations for 
GMRES using pseudoinverse and reorthogonalization (blue), 
GMRES using reorthogonalization (red), 
and RRGMRES (green) for a range symmetric inconsistent problem}
\label{fig:Ar-pinv-re2-gm-re2-rr-iryu}
\end{minipage}
\begin{minipage}{0.5\hsize}
\begin{center}
\includegraphics[keepaspectratio,scale=0.4]{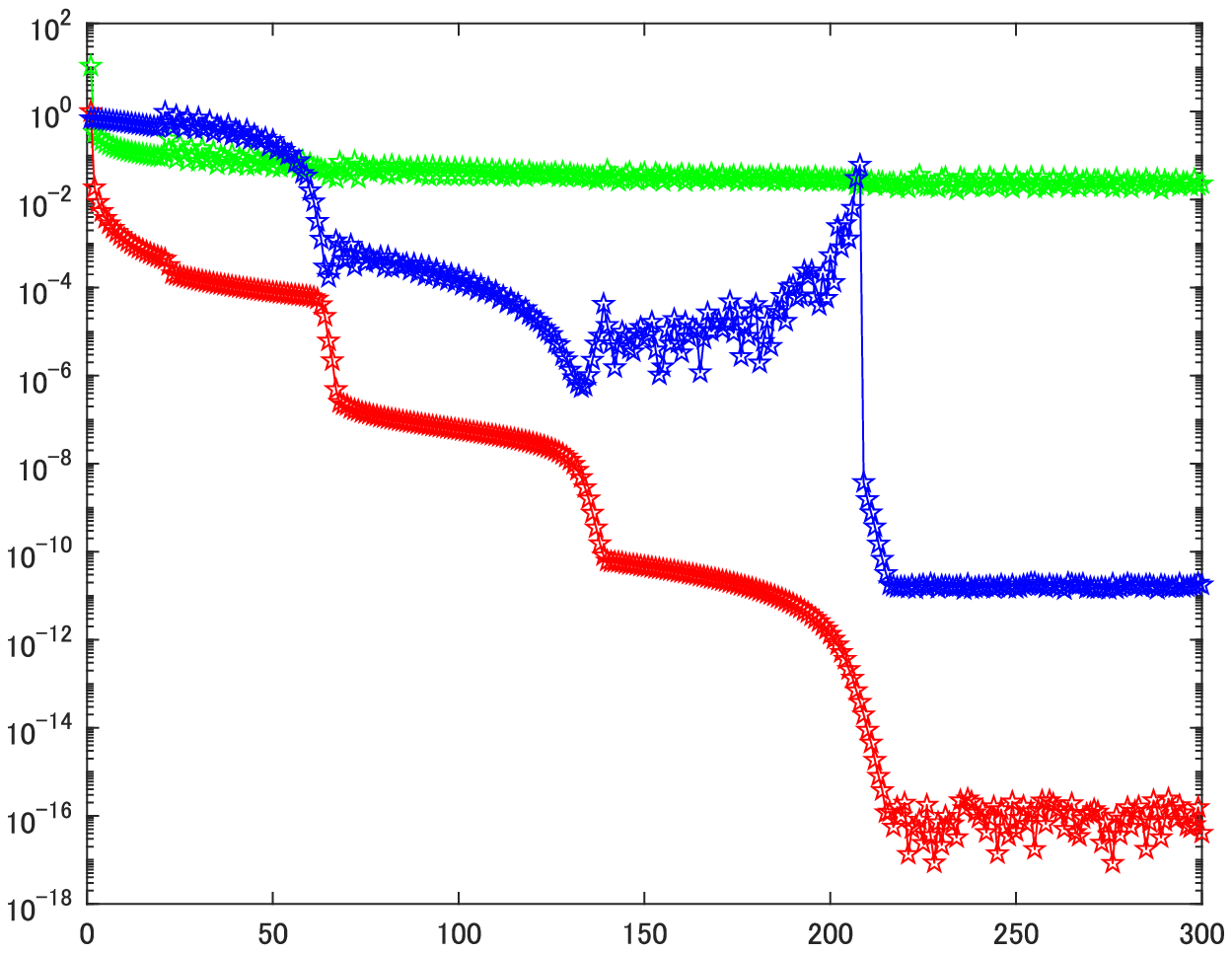}
\end{center}
\captionsetup{width=.95\linewidth}
\caption{$\displaystyle \frac{\|A^{\rm T}\vector{r}_{j}\|_{2}}{\|A^{\rm T}\vector{b}\|_{2}}$ (blue), 
$\frac{\sigma_{k}(H_{k+1,k})}{\sigma_{1}(H_{k+1,k})}$ (red) and 
$\frac{h_{j+1,j}}{\|H_{j,j}\|_{F}}$ (green) vs. number of iterations for 
GMRES using pseudoinverse and reorthogonalization for a range symmetric inconsistent problem}
\label{fig:Ar-sigma-h-pinv-re2-iryu}
\end{minipage}
\end{figure}

\begin{figure}[htbp]
\begin{minipage}{0.5\hsize}
\begin{center}
\includegraphics[keepaspectratio,scale=0.4]{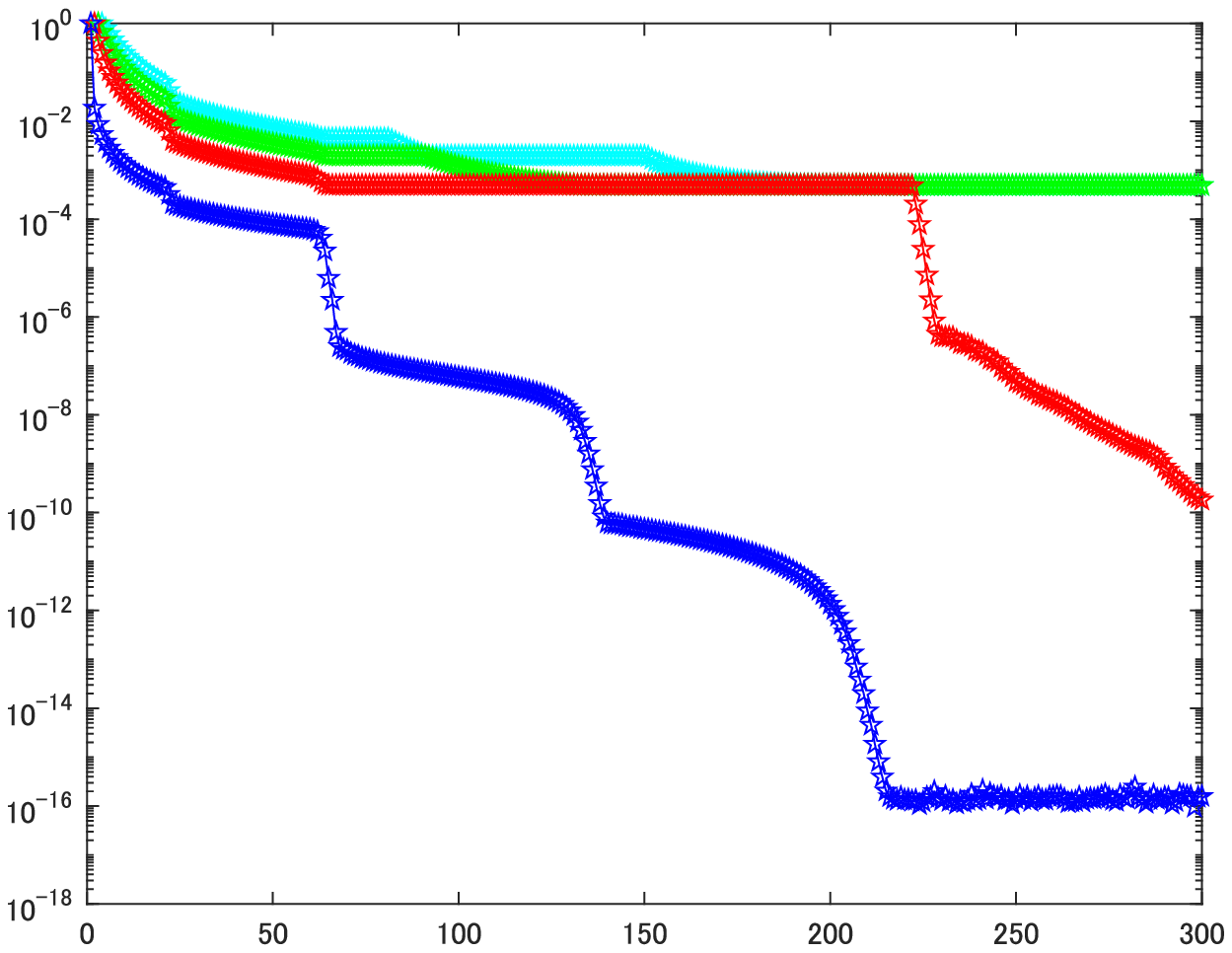}
\end{center}
\captionsetup{width=.95\linewidth}
\caption{
$\frac{\sigma_{k}(H_{k+1,k})}{\sigma_{1}(H_{k+1,k})}$(blue), 
$\frac{\sigma_{k-1}(H_{k+1,k})}{\sigma_{1}(H_{k+1,k})}$(red),
$\frac{\sigma_{k-2}(H_{k+1,k})}{\sigma_{1}(H_{k+1,k})}$(green) and 
$\frac{\sigma_{k-3}(H_{k+1,k})}{\sigma_{1}(H_{k+1,k})}$(cian) \\
vs. number of iterations for GMRES using \\ pseudoinverse 
for an inconsistent problem}
\label{fig:sig-min1234-gm-iryu}
\end{minipage}
\begin{minipage}{0.5\hsize}
\begin{center}
\includegraphics[keepaspectratio,scale=0.4]{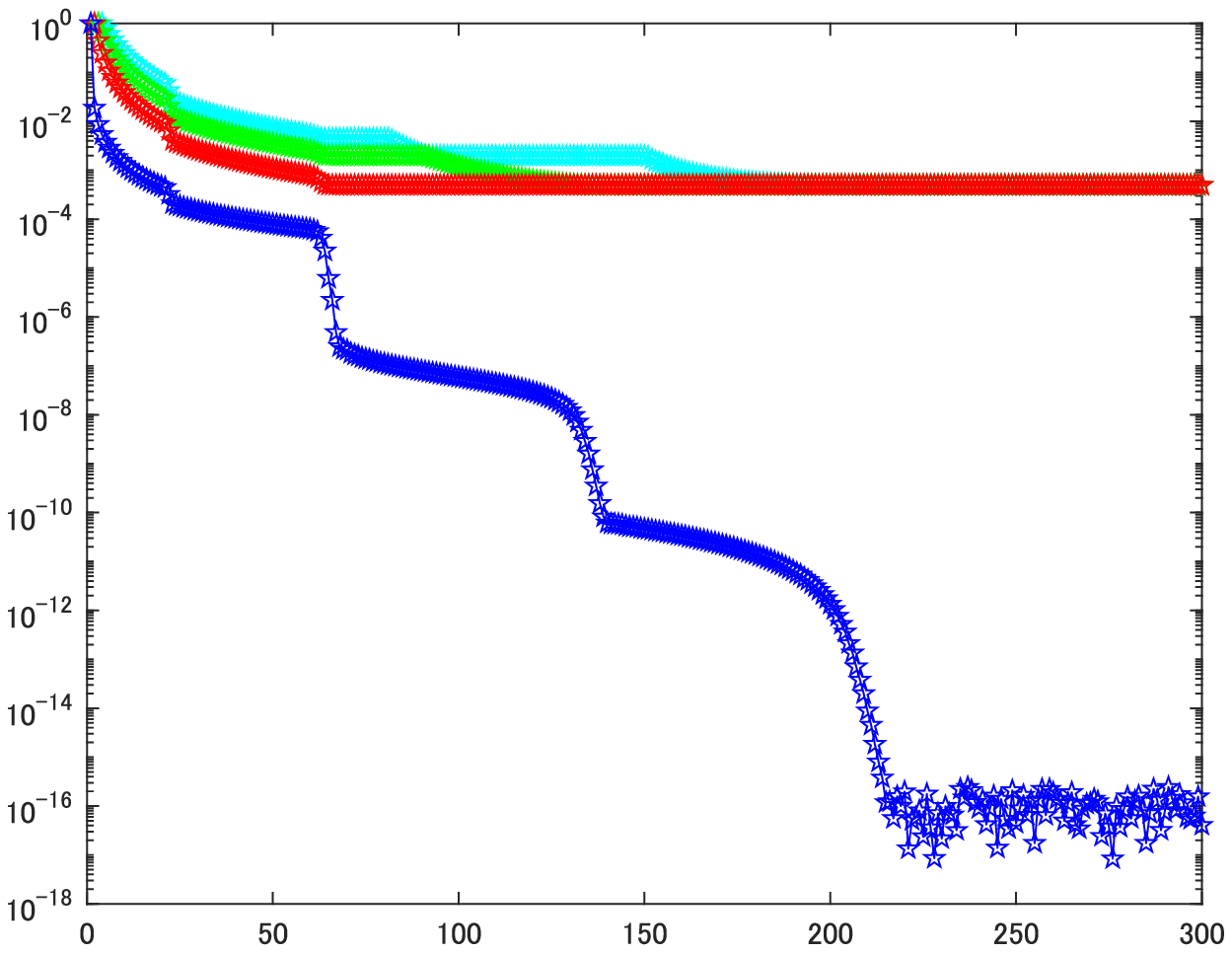}
\end{center}
\captionsetup{width=.95\linewidth}
\caption{
$\frac{\sigma_{k}(H_{k+1,k})}{\sigma_{1}(H_{k+1,k})}$(blue), 
$\frac{\sigma_{k-1}(H_{k+1,k})}{\sigma_{1}(H_{k+1,k})}$(red),
$\frac{\sigma_{k-2}(H_{k+1,k})}{\sigma_{1}(H_{k+1,k})}$(green) and 
$\frac{\sigma_{k-3}(H_{k+1,k})}{\sigma_{1}(H_{k+1,k})}$(cian) \\
vs. number of iterations for GMRES using \\ pseudoinverse and reorthogonalization
for an \\ inconsistent problem}
\label{fig:sig-min1234-gm-re2-iryu}
\end{minipage}
\end{figure}

Fig. \ref{fig:sig-min1234-gm-iryu} shows that 
$\frac{\sigma_{k-1}(H_{k+1,k})}{\sigma_{1}(H_{k+1,k})}$ 
of GMRES
using pseudoinverse
becomes small. However, it is not truncated by using the pseudoinverse. Thus, 
Fig. \ref{fig:Ar-pinv-gm-rr-iryu}
shows that
$\displaystyle \frac{\|A^{\rm T}\vector{r}_{j}\|_{2}}{\|A^{\rm T}\vector{b}\|_{2}}$
of this method increases after $\displaystyle \frac{\|A^{\rm T}\vector{r}_{j}\|_{2}}{\|A^{\rm T}\vector{b}\|_{2}}$ becomes smallest.
On the other hand, Fig. \ref{fig:sig-min1234-gm-re2-iryu} shows that 
$\frac{\sigma_{k-1}(H_{k+1,k})}{\sigma_{1}(H_{k+1,k})}$ of
GMRES
using
pseudoinverse and  
reorthogonalization 
is larger than $10^{-4}$ and stagnates. Thus, Fig. \ref{fig:Ar-pinv-re2-gm-re2-rr-iryu}
shows that
$\displaystyle \frac{\|A^{\rm T}\vector{r}_{j}\|_{2}}{\|A^{\rm T}\vector{b}\|_{2}}$
of GMRES using pseudoinverse and reorthogonalization does not increase
after $\displaystyle \frac{\|A^{\rm T}\vector{r}_{j}\|_{2}}{\|A^{\rm T}\vector{b}\|_{2}}$ becomes smallest.

\section{Concluding remarks}\label{sec:ConcL}
We derived the necessary and sufficient conditions for GMRES to determine 
a least squares solution of inconsistent and consistent range symmetric systems assuming exact arithmetic 
except for the computation of the elements of the Hessenberg matrix.
Then, we proposed using pseudoinverse to solve the Hessenberg systems in GMRES in order to improve
the numerical convergence for inconsistent systems.
Some numerical experiments on symmetric semidefinite inconsistent systems and a nonsymmetric
but range symmetric inconsistent system indicate that the method is 
effective and robust.
Moreover, we proposed GMRES using pseudoinverse and reorthogonalization
to further stabilize the convergence by suppressing the oscillation of the residual. 

\section{Acknowledgement}
We would like to thank Dr. Keiichi Morikuni for valuable discussions, and Professor Lothar Reichel for 
valuable remarks.



\bibliography{mybibfile}

\end{document}